\documentclass[a4paper,reqno]{amsart}

\usepackage[main=british]{babel}
\usepackage[T1]{fontenc}
\usepackage{microtype,amssymb,mathtools}
\usepackage{booktabs,upref}
\usepackage[hidelinks]{hyperref}

\newtheorem*{mainresult}{Main result}
\newtheorem{theorem}{Theorem}[section]
\newtheorem{proposition}[theorem]{Proposition}
\newtheorem{lemma}[theorem]{Lemma}
\newtheorem{corollary}[theorem]{Corollary}
\newtheorem{conjecture}[theorem]{Conjecture}
\newtheorem{question}[theorem]{Question}

\theoremstyle{definition}
\newtheorem{definition}[theorem]{Definition}

\theoremstyle{remark}
\newtheorem{remark}[theorem]{Remark}
\newtheorem{example}[theorem]{Example}

\numberwithin{equation}{section}

\newcommand{\any}{\,\cdot\,}
\newcommand{\hook}{\mathbin{\lrcorner}}

\newcommand{\cA}{\mathcal{A}}
\newcommand{\cC}{\mathcal{C}}

\newcommand{\bC}{\mathbb{C}}
\newcommand{\bR}{\mathbb{R}}
\newcommand{\bN}{\mathbb{N}}
\newcommand{\bZ}{\mathbb{Z}}

\newcommand{\lie}[1]{\mathfrak{#1}}
\newcommand{\mfa}{\lie{a}}
\newcommand{\mfg}{\lie{g}}
\newcommand{\mfh}{\lie{h}}
\newcommand{\mfk}{\lie{k}}
\newcommand{\mfp}{\lie{p}}
\newcommand{\mfr}{\lie{r}}
\newcommand{\aff}{\lie{aff}}
\newcommand{\gl}{\lie{gl}}
\newcommand{\sP}{\lie{sp}}
\newcommand{\un}{\lie{u}}

\newcommand{\spa}[1]{\mathrm{span}(#1)}
\newcommand{\cycl}{\mathrm{cyclic}}

\newcommand{\derg}{\mfg'}

\DeclareMathOperator{\re}{Re}
\DeclareMathOperator{\im}{Im}

\DeclareMathOperator{\End}{End}
\DeclareMathOperator{\Hom}{Hom}
\DeclareMathOperator{\ad}{ad}
\DeclareMathOperator{\id}{id}
\DeclareMathOperator{\inc}{inc}
\DeclareMathOperator{\tr}{tr}

\DeclarePairedDelimiter{\abs}{\lvert}{\rvert}
\DeclarePairedDelimiter{\norm}{\lVert}{\rVert}
\DeclarePairedDelimiter{\paren}{\lparen}{\rparen}
\DeclarePairedDelimiter{\Set}{\lbrace}{\rbrace}

\begin{document}

\title[Compatibility of balanced and SKT metrics]{Compatibility of
balanced and SKT metrics on two-step solvable Lie groups}

\author{Marco Freibert}

\address{Mathematisches Seminar\\
Christian-Albrechts-Universit\"at zu Kiel\\
Heinrich-Hecht-Platz 6\\
D-24118 Kiel\\
Germany}

\email{freibert@math.uni-kiel.de}

\author{Andrew Swann}

\address{Department of Mathematics and DIGIT\\
Aarhus University\\
Ny Munkegade 118, Bldg 1530\\
DK-8000 Aarhus C\\
Denmark}

\email{swann@math.au.dk}

\begin{abstract}
  It has been conjectured by Fino and Vezzoni that a compact complex
  manifold admitting both a compatible SKT and a compatible balanced
  metric also admits a compatible K\"ahler metric.
  Using the shear construction and classification results for two-step
  solvable SKT Lie algebras from our previous work, we prove this
  conjecture for compact two-step solvmanifolds endowed with an
  invariant complex structure which is either (a)~of pure type or
  (b)~of dimension six.
  In contrast, we provide two counterexamples for a natural
  generalisation of this conjecture in the homogeneous invariant
  setting.
  As part of the work, we obtain further classification results for
  invariant SKT, balanced and K\"ahler structures on two-step solvable
  Lie groups.
  In particular, we give the full classification of left-invariant SKT
  structures on two-step solvable Lie groups in dimension six.
\end{abstract}

\maketitle

\tableofcontents

\section{Introduction}

Hermitian geometry has been a very active field of study for several
decades.
Historically, most focus has been on K\"ahler manifolds and many
important result have been obtained for these manifolds, including
several severe topological restrictions in the compact case.
These topological restrictions show that most compact complex
manifolds do not admit a K\"ahler structure and this has led to rising
interest in generalisations of K\"ahler manifolds.
In particular, two types of non-K\"ahler Hermitian manifolds
\( (M,g,J,\sigma) \) have intensively been investigated, namely
\emph{strong K\"ahler with torsion (SKT)} and \emph{balanced}
manifolds, which are characterised by \( dJ^*d\sigma=0 \) or
\( \delta_g \sigma=0 \), respectively.

Interest in SKT manifolds stems from various sources.
First of all, these manifolds are precisely those almost Hermitian
manifolds for which there exists a compatible
connection~\( \nabla^B \) with totally skew-symmetric
torsion~\( T^B \) that, when considered as a three-form, is closed.
It is because of this property that they occur in physics in the
context of supersymmetric theories, see for example~\cite{GHR,HP,St}.
Secondly, any conformal class on a compact complex surface contains an
SKT metric~\cite{Gau}, a property which is no longer true in higher
dimensions.
Moreover, any compact even-dimensional Lie group admits a
left-invariant SKT structure~\cite{SSTvP,MaSw}.

Balanced metrics are of interest since they occur naturally on several
types of complex manifolds: for example, any unimodular complex Lie
group~\cite{AG} admits a compatible left-invariant balanced metric and
any compact complex manifold which is bimeromorphic to a compact
K\"ahler manifold admits a compatible balanced metric~\cite{AB}.
Furthermore, they are an important ingredient in the Strominger system
from physics \cite{FY,Fe}.

The SKT and balanced conditions for a fixed Hermitian structure are
known to be mutually exclusive in the sense that any compatible metric
\( g \) on a complex manifold \( (M,J) \) which is both SKT and
balanced has to be K\"ahler~\cite{AI}.
More generally, Fino and Vezzoni conjectured

\begin{conjecture}[{\cite[Problem~3]{FiV1}, \cite[Conjecture]{FiV2}}]
  \label{conj:FV}
  Any compact complex manifold \( (M,J) \) which admits a compatible
  SKT metric and admits a compatible balanced metric also admits a
  compatible K\"ahler metric.
\end{conjecture}

This conjecture has been confirmed in some special cases, including
twistor spaces of compact anti-self-dual Riemannian
manifolds~\cite{V}, non-K\"ahler manifolds belonging to the Fujiki
class~\( \cC \)~\cite{Ch} and for left-invariant complex structures on
compact semi-simple Lie groups~\cite{FGV,P}.
We will extend the latter result to all compact even-dimensional Lie
groups in Theorem~\ref{thm:compact}.

Recall that nil- and solvmanifolds are manifolds of the form
\( M = \Gamma\backslash G \), with \( G \) nilpotent or solvable,
respectively, and \( \Gamma \) a discrete subgroup.
One says that a complex structure~\( J \) on \( \Gamma\backslash G \)
is invariant if it pulls-back to a left-invariant complex structure
on~\( G \).
Now Conjecture~\ref{conj:FV} has been proved for all compact
nilmanifolds with invariant complex structure (by \cite{FiV2} combined
with \cite{AN}), and for the following classes of compact
solvmanifolds with invariant complex structure: six-dimensional
solvmanifolds with holomorphically trivial canonical
bundle~\cite{FiV1}, almost Abelian solvmanifolds~\cite{FiP1},
Oeljeklaus-Toma manifolds~\cite{Ot}, regular complex structures on
non-compact semi-simple Lie groups~\cite{G}, and special types of
invariant metrics on almost nilpotent solvmanifolds for which the
associated Lie algebra has a nilradical with one-dimensional
commutator~\cite{FiP2}.

Actually, compactness of \( \Gamma\backslash G \) implies that \( G \)
is unimodular, and in all of the above cases of
\( \Gamma\backslash G \) with invariant~\( J \) the results are proved
by just considering all left-invariant structures on~\( G \).
This is possible since \cite{FiG,U} showed that once \( J \) is
left-invariant, any compatible SKT or balanced metric on
\( \Gamma\backslash G \) may be averaged to a left-invariant metric of
the same type.
Hence, the following question has a positive answer for these Lie
groups~\( G \).

\begin{question}
  \label{quest:G}
  Let \( G \) be a unimodular Lie group with a left-invariant complex
  structure~\( J \).
  If \( (G,J) \) admits a left-invariant compatible SKT metric and
  admits a left-invariant compatible balanced metric, does \( (G,J) \)
  also admit a left-invariant compatible K\"ahler metric?
\end{question}

In this paper, we consider Question~\ref{quest:G} for \( G \) two-step
solvable.
We have imposed that all the metrics considered be left-invariant,
since any solvable group admits a K\"ahler metric, see
Proposition~\ref{pro:Kahleronsolvable}, however this metric is not
necessarily invariant.

We will give two examples of two-step solvable Lie groups to show the
condition of unimodularity is necessary.
In particular, Examples \ref{ex:counterexampletypeI}
and~\ref{ex:counterexampletypeIII} are the first known examples of Lie
groups with left-invariant complex structure admitting both
left-invariant SKT metrics and left-invariant balanced metrics,
without admitting left-invariant K\"ahler metrics.
As these examples are not unimodular, they do not give counterexamples
to Conjecture~\ref{conj:FV}.

In contrast, we give a positive answer to Question~\ref{quest:G} for
all two-step solvable Lie groups \( G \) endowed with an invariant
complex structure \( J \) in the following situations: (a) \( J \) is
of \emph{pure type}, (b) \( G \) is of dimension~\( 6 \).
Here ``pure type'' means one of three summands in a natural
decomposition of the Lie algebra \( \mfg \) vanishes.
Writing \( \derg = [\mfg,\mfg] \) for the derived algebra, we have the
following pure types: (I)~\( \derg \cap J\derg = 0 \),
(II)~\( \derg = J\derg \), (III)~\( \derg + J\derg = \mfg \).

Our approach, is to build on our study~\cite{FrSw2} of SKT structures
on two-step solvable Lie groups.
The individual cases are proved in
Theorems~\ref{th:compatibilitypuretypeI}, \ref{th:SKTpuretypeII} and
\ref{th:compatibilitypuretypeIII} for the three pure types.
For the six-dimensional case we first complete the classification of
SKT structures on two-step solvable Lie groups in
Theorem~\ref{th:classification6dSKT} and then apply this to
Question~\ref{quest:G} in Theorem~\ref{th:compatibility6d}.

From the remarks on averaging above, we then get

\begin{mainresult}
  Let \( \Gamma\backslash G \) be a compact two-step solvmanifold and
  with an invariant complex structure \( J \).
  Then Conjecture~\ref{conj:FV} holds if either \( G \) is
  six-dimensional, or \( (\Gamma\backslash G,J) \) is of pure type.
\end{mainresult}

The paper is organised as follows.
Definitions from Hermitian geometry, the notions of pure type, our
approach to two-step solvable Lie algebras via the shear construction,
and notation for concrete Lie algebras are summarised
in~\S\ref{sec:prelims}.
We then derive some general results for two-step solvable Lie
algebras, in the case that they have either a balanced structure
\S\ref{subsec:balanced}, or a K\"ahler structure
\S\ref{subsec:Kahler}, in the latter case obtaining more detailed
information than the general structural results of \cite{D}.
Then in \S\ref{sec:compatibility}, we prove the main theorems of the
paper, as described above.
One consequence is an explicit list of two-step solvable K\"ahler Lie
algebras in dimension six, see~Corollary~\ref{th:6dKahler}.

\section{Preliminaries}
\label{sec:prelims}

\subsection{Non-K\"ahler Hermitian geometry}
\label{subsec:nonkahlergeometry}

First, we recall the basic definitions.
We write \( (M,J) \) for a complex manifold, so \( J \) is an
\emph{integrable} complex structure.
A metric \( g \) is compatible with \( J \) if
\( g(J\any,J\any) = g(\any,\any) \), and then the triple \( (M,g,J) \)
is called a Hermitian structure.
We write \( \sigma \coloneqq g(J\any,\any) \) for the fundamental
two-form.

\begin{definition}
  A Hermitian manifold \( (M,g,J) \) of dimension \( 2n \) is
  \begin{enumerate}
  \item[(i)] \emph{K\"ahler} if \( d\sigma = 0 \),
  \item[(ii)] \emph{balanced} if \( d (\sigma^{n-1})=0 \),
  \item[(iii)] \emph{strong K\"ahler with torsion (SKT)} manifold if
    \( dJ^*d\sigma=0 \).
  \end{enumerate}
\end{definition}

The following result is given in \cite[Remark~1]{AI}.

\begin{proposition}
  \label{pro:balanced+SKT=Kahler}
  Let \( (M,g,J) \) be a Hermitian manifold which is both balanced and
  SKT.  Then \( (M,g,J) \) is a K\"ahler manifold. \qed
\end{proposition}

Next, consider a simply-connected Lie group \( G \) which admits a
cocompact discrete subgroup~\( \Gamma \) and consider the compact manifold
\( M \coloneqq \Gamma\backslash G \).
Any left-invariant tensor-field on~\( G \) may be pushed down to a
tensor field on~\( M \).
The resulting tensor fields on~\( M \) are said to be
\emph{invariant}.

By using averaging one has the following result of \cite{FiG}
and~\cite{U}.

\begin{proposition}
  \label{pro:invariantbalancedorSKT}
  Suppose \( M \coloneqq \Gamma \backslash G \) is a compact manifold
  that is the quotient of a simply-connected Lie group~\( G \) by a
  discrete subgroup~\( \Gamma \).

  Suppose \( J \) is an invariant complex structure on~\( M \).
  If \( (M,J) \) admits a compatible metric that is balanced or SKT,
  then it also admits a compatible invariant balanced or SKT metric,
  respectively. \qed
\end{proposition}

Now consider left-invariant structures on the simply-connected Lie
group \( G \) directly and identify them with the corresponding
structures on the associated Lie algebra~\( \mfg \).
We may then also speak of \emph{Hermitian}, \emph{balanced} or
\emph{SKT} Lie algebras.
Throughout we will assume \( \mfg \ne 0 \) and write
\( \dim\mfg = 2n \) for the dimension of \( \mfg \) over~\( \bR \).

\begin{definition}
  \label{def:V}
  For a Hermitian Lie algebra \( (\mfg, g, J) \) we introduce the
  following vector subspaces.
  We set \( \derg_J = \derg\cap J\derg \) to be the maximal complex
  subspace of the derived algebra~\( \derg = [\mfg,\mfg]\).
  Moreover, we let \( \derg_r \) be the orthogonal complement of
  \( \derg_J \) in \( \derg \) and then define
  \( V_r \coloneqq \derg_r\oplus J\derg_r \).
  Note that this direct sum is not orthogonal in general.
  Observe that now
  \begin{equation*}
    \derg+J\derg=\derg_J\oplus V_r
  \end{equation*}
  and define \( V_J \) to be the orthogonal complement of
  \( \derg+J\derg \) in \( \mfg \).
  We then have a vector space orthogonal direct sum decomposition
  \begin{equation}
    \label{eq:g-V}
    \mfg=\derg_J\oplus V_r\oplus V_J.
  \end{equation}
  of \( \mfg \) by spaces that are preserved by~\( J \).
  We define \( s,r,\ell \in \bN \) by
  \begin{equation*}
    2s \coloneqq \dim(\derg_J),\qquad 2r \coloneqq \dim(V_r),\qquad
    2\ell \coloneqq \dim(V_J)
  \end{equation*}
  and use this notation throughout the article.
  Note that so \( s+r+\ell = n \) and that these numbers depend only
  on \( (\mfg,J) \), and not on the metric \( g \), since
  \( 2r = \dim(V_r) = \dim(\derg+J\derg)-\dim(\derg_J) =
  \dim(\derg+J\derg)-2s \) and \( \ell =n-r-s \).

  We call \( (\mfg,g,J) \) or \( (\mfg,J) \) of
  \begin{enumerate}
  \item \emph{pure type~I} if \( \derg_J=0 \), i.e.\ \( \derg \) is
    totally real,
  \item \emph{pure type~II} if \( V_r=0 \) or, equivalently,
    \( \derg_r=0 \), i.e.\ if \( \derg \) is complex, and
  \item \emph{pure type~III} if \( V_J=0 \), i.e.\ if
    \( \mfg=\derg+J\derg \).
  \end{enumerate}
  If one of these conditions hold, we simply say that \( (\mfg,g,J) \)
  or \( (\mfg,J) \) is of \emph{pure type}.
\end{definition}

\subsection{Complex shears of \texorpdfstring{\( \bR^{2n} \)}{R 2n}}
\label{subsec:complexshear}

Although the shear construction is defined for arbitrary manifolds, we
will just need the version for Lie groups and algebras as presented
in~\cite{FrSw2}.
The motivating example is as follows.
Let \( H,P,K \) be simply-connected Lie groups with
\( \dim(H) = \dim(K) \) and whose associated Lie algebras are related
by surjective Lie algebra homomorphisms
\( \mfh\leftarrow \mfp \rightarrow \mfk \) with Abelian kernels.
We then have two Abelian Lie algebras \( \hat{\mfa}_P \), \( \mfa_P \)
of the same dimension, and two Lie algebra extensions of the form
\begin{equation*}
  \mfh \twoheadleftarrow \mfp \hookleftarrow \mfa_P
  \quad\text{and}\quad
  \hat{\mfa}_P \hookrightarrow \mfp \twoheadrightarrow \mfk.
\end{equation*}
The Lie algebra \( \mfp \) and the \emph{shear algebra} \( \mfk \) may
constructed from certain ``shear data'' on~\( \mfh \): two Abelian Lie
algebras \( \mfa_{H} \) and \( \mfa_{P} \), a Lie algebra monomorphism
\( \xi\colon \mfa_H\rightarrow \mfh \), a two-form
\( \omega\in \Lambda^2 \mfh^*\otimes \mfa_P \) on \( \mfh \) with
values in~\( \mfa_P \), a representation
\( \eta\in \mfh^*\otimes \Hom(\mfa_P) \) and a Lie algebra isomorphism
\( a\colon \mfa_H\rightarrow \mfa_P \), with certain compatibility.
The vector space \( \mfp \) is then \( \mfh\oplus \mfa_P \) and the
Lie bracket is specified by
\( [X,Y]_{\mfp} \coloneqq [X,Y]_{\mfh} - \omega(X,Y) \),
\( [X,Z]_{\mfp} \coloneqq \eta(X)(Z) \), for \( X,Y\in \mfh \) and
\( Z\in \mfa_P \), and that \( \mfa_{P} \) is an Abelian ideal.
Writing \( \rho\colon \mfa_P\rightarrow \mfp=\mfh\oplus \mfa_P \) for
the natural inclusion, the map
\begin{equation*}
  \mathring{\xi}\colon \mfa_H\rightarrow \mfp,\qquad \mathring{\xi}
  \coloneqq \xi+\rho\circ a
\end{equation*}
is a Lie algebra monomorphism and \( \mathring{\xi}(\mfa_H) \) is an
ideal in~\( \mfp \).
The shear algebra~\( \mfk \) is then the
quotient~\( \mfp/\mathring{\xi}(\mfa_H) \).
As vector spaces, \( \mfh \)~is a summand of
\( \mfp = \mfh\oplus \mfa_P \) and is isomorphic the shear
algebra~\( \mfk = \mfp/\mathring{\xi}(\mfa_{H}) \) via the projection.
In this way, tensors on \( \mfh \) may be transferred to tensors
on~\( \mfk \).

For studying two-step solvable algebras, we may take
\( \mfh = \bR^N \) Abelian, \( \mfa = \mfa_{H} = \mfa_{P} \) a
subalgebra of~\( \mfh = \bR^{N} \), \( \xi = \inc \), the inclusion
map, and \( a = \id_{\mfa} \), see~\cite{FrSw2}.
The remaining shear data is
\( \omega\in \Lambda^2 \mfh^*\otimes \mfa \) and
\( \eta \in \mfh^* \otimes \Hom(\mfa) \).
Compatibility gives \( \eta(X)(Z)=-\omega(X,Y) \) for \( X\in \mfh \),
\( Z\in \mfa \) and \( \omega|_{\Lambda^2\mfa^*}=0 \).
Thus \( \mfa \) and \( \omega \) determine the entire shear data.

\begin{definition}
  A pair \( (\mfa,\omega) \) consisting of a subspace \( \mfa \) of
  \( \bR^N \) and a two-form
  \( \omega\in \Lambda^2(\bR^N)^* \otimes \mfa \) with
  \( \omega|_{\Lambda^2 \mfa}=0 \) is called \emph{pre-shear data} (on
  \( \bR^N \)).
\end{definition}

Now suppose that \( N=2n \).
As \( \bR^{2n} \) is Abelian any \( J \in \End(\bR^{2n}) \) with
\( J^{2} = -\id \) defines an (integrable) complex structure.
Combining \cite[Lemmas 2.1 and~3.1]{FrSw2} and
\cite[Proposition~2.5]{FrSw1} gives

\begin{proposition}
  \label{pro:complexsheardata}
  Let \( J \) be a complex structure on the Lie algebra \( \bR^{2n} \)
  and let \( (\mfa,\omega) \) be pre-shear data on \( \bR^{2n} \).
  Then the shear \( (\mfg,J_{\mfg}) \) of~\( (\bR^{2n},J) \) is a Lie
  algebra~\( \mfg \) endowed with a complex structure \( J_{\mfg} \)
  if and only if
  \begin{equation}
    \label{eq:complexsheardata}
    \cA(\omega(\omega(\any,\any),\any))=0,\qquad
    J^*\omega=\omega-J\circ J.\omega,
  \end{equation}
  where \( J.\omega = - \omega(J\any,\any) - \omega(\any,J\any) \) and
  \( \cA \) is anti-symmetrisation.
  Moreover, every two-step solvable Lie algebra \( \mfg \) with a
  complex structure~\( J_{\mfg} \) may be obtained in this way. \qed
\end{proposition}

\begin{definition}
  Pre-shear data \( (\mfa,\omega) \) on \( (\bR^{2n},J) \) that
  satisfies~\eqref{eq:complexsheardata} will be called \emph{complex
  shear data}.
\end{definition}

To describe certain consequences of~\eqref{eq:complexsheardata}, we
need some further notation.
On~\( \bR^{2n} \), let \( (\mfa,\omega) \) be pre-shear data and let
\( J \)~be a complex structure.
Any compatible metric~\( g \) on \( \bR^{2n} \) then makes
\( (\bR^{2n},g,J) \) into a K\"ahler Lie algebra.

Set \( \mfa_J \coloneqq \mfa\cap J\mfa \), and let \( \mfa_r \) be the
orthogonal complement of \( \mfa_J \) in~\( \mfa \).
Put \( U_r \coloneqq \mfa_r \oplus J\mfa_r \) and let \( U_J \) be the
orthogonal complement of \( \mfa+J\mfa=\mfa_J\oplus U_r \)
in~\( \bR^{2n} \).

For each \( X\in \mfa \), we define
\begin{equation*}
  A_X \coloneqq \omega(JX,\any)|_{\mfa}\in \End(\mfa)
\end{equation*}
and decompose \( A_X \) according to the splitting
\( \mfa=\mfa_J\oplus \mfa_r \) as
\begin{equation*}
  \begin{split}
    A_X
    &=K_X+G_X+H_X+F_X\\
    &\in \End(\mfa_J) \oplus \Hom(\mfa_J,\mfa_r)
      \oplus \Hom(\mfa_r,\mfa_J) \oplus \End(\mfa_r).
  \end{split}
\end{equation*}
We have associated bilinear maps
\( f\colon \mfa_r\otimes \mfa_r\rightarrow \mfa_r \) and
\( h\colon \mfa_r\otimes \mfa_r\rightarrow \mfa_J \) given by
\begin{equation*}
  f(X,\hat{X}) \coloneqq F_{X}(\hat{X}),\qquad h(X,\hat{X}) \coloneqq
  H_{X}(\hat{X}).
\end{equation*}
Moreover, for each \( Z\in U_J \), we set
\begin{equation*}
  B_Z \coloneqq \omega(Z,\any)|_{\mfa} \in \End(\mfa).
\end{equation*}

\begin{lemma}
  \label{le:conscomplexsheardata}
  On the Abelian Lie algebra \( \bR^{2n} \), let \( (g,J) \) be a
  K\"ahler structure and let \( (\mfa,\omega) \) be complex shear
  data.  Then
  \begin{enumerate}
  \item[\textup{(i)}] \( G_X=0 \) for all \( X\in \mfa_r \),
  \item[\textup{(ii)}] \( [J,K_X]=0 \) for all \( X\in \mfa_r \),
  \item[\textup{(iii)}] \( f \) is symmetric,
  \item[\textup{(iv)}] for all \( X,\hat{X}\in \mfa_r \), we have
    \begin{equation*}
      \omega(JX,J\hat{X})\in \mfa_J \quad\text{and}\quad
      \omega(JX,J\hat{X})
      = J(h(X,\hat{X})-h(\hat{X},X)),
    \end{equation*}
  \item[\textup{(v)}] for all \( \tilde{X},\hat{X}\in \mfa_r \) and
    all \( \tilde{Z},\hat{Z}\in U_J \), we have
    \begin{equation*}
      [A_{\tilde{X}}, A_{\hat{X}}] = 0 = [A_{\tilde{X}},
      B_{\tilde{Z}}] = [B_{\tilde{Z}},B_{\hat{Z}}]
      \quad\text{and}\quad [K_{\tilde{X}}, K_{\hat{X}}]=0.
    \end{equation*}
    Moreover,
    \begin{equation*}
      \omega^{r}(JZ,JX) = \omega^{r}(Z,X),
    \end{equation*}
    where \( \omega^{r} \) is the part of \( \omega \) that takes
    values in~\( \mfa_{r} \).
  \end{enumerate}
\end{lemma}

\begin{proof}
  Parts (i)--(iv) and the final part of~(v) may be found
  in~\cite[Lemma 3.3]{FrSw2}.
  The rest of part~(v) follows from the first equation
  in~\eqref{eq:complexsheardata} evaluated on one element
  of~\( \mfa \) and two elements of \( J\mfa_r\oplus U_J \), and from
  the fact that the endomorphism \( A_X \) is block upper triangular
  by part~(i).
\end{proof}

Finally, we recall the following result from~\cite[Lemma 3.1]{FrSw2}.

\begin{lemma}
  \label{le:SKTsheardata}
  On the Abelian Lie algebra \( \bR^{2n} \), let \( (g,J) \) be a
  K\"ahler structure and let \( (\mfa,\omega) \) be complex shear
  data.
  Then the shear \( (\mfg, g_{\mfg}, J_{\mfg}) \) of
  \( (\bR^{2n},g,J) \) is an SKT Lie algebra if and only if
  \begin{equation}
    \label{eq:SKTsheardata}
    \cA\paren[\Big]{g(J^*\omega(\any,\any),
    \omega(\any,\any)) +
    2g(J^*\omega(\omega(\any,\any),\any),\any)} = 0,
  \end{equation}
  where \( \cA \) is anti-symmetrisation. \qed
\end{lemma}

\begin{remark}
  In the rest of the article, when we consider a two-step solvable
  (almost) Hermitian Lie algebra \( (\mfg, g_{\mfg}, J_{\mfg}) \), we
  will regard it as being obtained by appropriate pre-shear data
  \( (\mfa,\omega) \) from a flat K\"ahler structure \( (g,J) \)
  on~\( \bR^{2n} \).
  The identification of \( \mfg = \bR^{2n} \) as vector spaces,
  identifies \( g_{\mfg} \) with \( g \) and \( J_{\mfg} \)
  with~\( J \).
  Note that this also gives \( -\omega = [\any,\any]_{\mfg} \) and
  that \( \derg = \im\omega \).
  We can then without loss of generality identify \( \mfa \)
  with~\( \derg \).
\end{remark}

\subsection{Lie algebra notation}
\label{sec:lie-algebra-notation}

When we need to specify concrete Lie algebras, we will often use
Salamon's notation~\cite{Sa}.
If \( e_{1},\dots,e_{n} \) is a basis for~\( \mfg \) with dual basis
\( e^{1},\dots,e^{n} \), then
\( de^{i}(e_{j},e_{k})= - e^{i}([e_{j},e_{k}]) \) and the algebra is
specified by listing the differentials \( (de^{1},\dots,de^{n}) \),
but writing for example
\( 3(e^{1}\wedge e^{2}-e^{4}\wedge e^{6}) = 3(e^{12}-e^{46}) \) as
\( 3.(12-46) \).

In Table~\ref{table_LAs}, we use this notation to list the Lie
algebras of this article that have standard names, with notation
coming from~\cite{Bi,ABDO,Mu1,Mu2,Tu}.

\begin{table}
  \centering
  \begin{tabular}{@{\vrule width 0pt height 2.5ex depth
    1.5ex}*{4}{l}@{}}
    \toprule
    \( \mfg \) & dim & differentials & \\
    \midrule
    \( \aff_{\bR}\) & \( 2 \) & \( (0,21) \) & \\
    \( \mfh_{3} \) & \( 3 \) & \( (0,0,21) \) & \\
    \( \mfr'_{3,\lambda} \) & \( 3 \)
                     & \( (0,\lambda.21+31,-21+\lambda.31) \)
                                     & \( \lambda\geqslant 0 \) \\
    \( \mfr_{4,\mu,\lambda} \) & \( 4 \)
                     & \( (0,21,\mu.31,\lambda.41) \)
                                     & \( 0<\abs{\lambda}\leqslant
                                       \abs{\mu}\leqslant 1 \) \\
    \( \mfr'_{4,\mu,\lambda} \) & \( 4 \)
                     & \( (0,\mu.21,\lambda.31+41,-31+\lambda.41) \)
                                     & \( \mu>0 \) \\
    \( \mfg_{5,17}^{\alpha,\beta,\gamma} \) & \( 5 \)
                     & \(
                       \begin{array}[t]{@{}l@{}}
                         (0,\alpha.21+31,-21+\alpha.31,\\
                         \qquad\beta.41+\gamma.51,
                         -\gamma.41+\alpha.51)
                       \end{array} \)
               & \( \alpha \geqslant 0 \), \( \gamma\neq 0 \) \\
    \( \mfg_{6,11}^{\alpha,\beta,\gamma,\delta} \) & \( 6 \)
                     & \(
                       \begin{array}[t]{@{}l@{}}
                         (0,\alpha.21,\beta.31+41,-31+\beta.41,\\
                         \qquad\quad
                         \gamma.51+\delta.61,-\delta.51+\gamma.61)
                       \end{array} \)
               & \( \alpha \delta\neq 0 \) \\
    \( N_{6,1}^{\alpha,\beta,\gamma,\delta} \) & 6
                     & \(
                       \begin{array}[t]{@{}l@{}}
                         (\alpha.15+\beta.16, \gamma.25+\delta.26,\\
                         35, 46, 0, 0)
                       \end{array} \)
               & \( \alpha\beta\ne0 \), \( (\gamma,\delta) \ne (0,0)
                 \)\\
    \( N_{6,14}^{\alpha,\beta,\gamma} \) & 6
                     & \(
                       \begin{array}[t]{@{}l@{}}
                         (\alpha.15+\beta.16,26,\\
                         \qquad\gamma.35-45,\gamma.45+35,0,0)
                       \end{array} \)
               & \(\alpha\beta\neq 0\) \\
    \bottomrule
  \end{tabular}
  \medskip
  \caption{Notation for certain Lie algebras}
  \label{table_LAs}
\end{table}

\section{Balanced and K\"ahler geometry}
\label{sec:balancedkahler}

\subsection{Balanced Lie algebras}
\label{subsec:balanced}

\begin{lemma}
  \label{le:balancedsheardata}
  Let \( (\bR^{2n},g,J) \) be a flat K\"ahler structure on the Lie
  algebra \( \bR^{2n} \) with associated K\"ahler form~\( \sigma \).
  Then the shear \( (\mfg,g_{\mfg},J_{\mfg}) \) of
  \( (\bR^{2n},g,J) \) by complex shear data \( (\mfg,\omega) \) is
  balanced if and only if
  \begin{equation}
    \label{eq:balancedsheardata}
    \cA\paren[\big]{\sigma(\omega(\any,\any),\any)\wedge \sigma^{n-2}}
    = 0.
  \end{equation}
\end{lemma}

\begin{proof}
  Since \( \xi=\inc\colon \mfa\rightarrow \bR^{2n} \) and
  \( a = \id_{\mfa} \), \cite[Corollary 3.11]{FrSw1} implies
  \begin{equation*}
    \begin{split}
      d_{\mfg} \sigma_{\mfg}^{n-1}
      &= d_{\bR^{2n}} \sigma^{n-1} - (\xi\circ a^{-1} \hook
        \sigma^{n-1}) \wedge \omega
        = - (n-1)\,(\inc\hook \sigma) \wedge \omega\wedge
        \sigma^{n-2}\\
      &=-(n-1)\, \cA\paren[\big]{\sigma(\omega(\any,\any),\any)\wedge
        \sigma^{n-2}}.
    \end{split}
  \end{equation*}
  So \( (\mfg,g_{\mfg},J_{\mfg}) \) is balanced, i.e.\
  \( d_{\mfg}\sigma_{\mfg}^{n-1}=0 \), if and only
  if~\eqref{eq:balancedsheardata} holds.
\end{proof}

\begin{proposition}
  \label{pro:balanced}
  Let \( (\mfg,g,J) \) be a Hermitian two-step solvable Lie algebra
  and write \( \mfg = \derg_{J}\oplus V_{r}\oplus V_{J} \) as
  in~\eqref{eq:g-V}.  Then \( (g,J) \) is balanced if and only if
  \begin{equation*}
    \tr(\ad(Z))=0
  \end{equation*}
  for all \( Z\in V_J \), and there exist unitary bases
  \( X_1,\dots,X_{2r} \) of \( (V_r,g,J) \) and
  \( Z_1,\dots,Z_{2\ell} \) of \( (V_J,g,J) \) such that the element
  \begin{equation}
    \label{eq:C}
    C = \sum_{i=1}^r [X_{2i-1},X_{2i}] +
    \sum_{j=1}^\ell [Z_{2j-1},Z_{2j}]
  \end{equation}
  is orthogonal to \( \derg_{J} \) and satisfies
  \begin{equation*}
    \tr(\ad(X)) = -\sigma \paren{C, X},
  \end{equation*}
  for any \( X\in V_{r} \).
\end{proposition}

\begin{proof}
  Fix unitary bases \( Y_1,Y_{2} = JY_{1},\dots,Y_{2s} \) of
  \( (\derg_J,g,J) \), \( X_1,\dots,X_{2r} \) of \( (V_r,g,J) \) and
  \( Z_1,\dots,Z_{2\ell} \) of \( (V_J,g,J) \).
  We have to check~\eqref{eq:balancedsheardata} for all combinations
  of \( 2n-1 \) vectors from the basis
  \( Y_1,\dots,Y_{2s},X_1,\dots,X_{2r},Z_1,\dots,Z_{2\ell} \) of
  \( \mfg \).
  There are thus three cases, corresponding to omitting one
  vector~\( W \) from \( \derg_J \), \( V_r \) or \( V_J \).

  First, if we omit
  \( W \in \Set{Z_{1},\dots,Z_{2\ell}} \subset V_{J} \), we can
  relabel our bases, so \( W = Z_{2r} \) and put \( Z = Z_{2r-1} \).
  Then \eqref{eq:balancedsheardata} only has non-zero contributions
  each \( \sigma \) is evaluated on pairs \( A, JA \), implying
  \( Z \) is then one of the arguments of
  \( \sigma(\omega(\any,\any),\any) \).
  But \( \im(\omega) \perp V_J \), so \( Z \) is an argument of
  \( \omega \) and the only contributions
  to~\eqref{eq:balancedsheardata} are
  \begin{equation*}
    \begin{split}
      0&=\sum_{j=1}^s\sigma(\omega(Z,Y_{2j-1}),Y_{2j}) -
         \sigma(\omega(Z,Y_{2j}),Y_{2j-1})
      \\
       &\qquad
         +\sum_{k=1}^r \sigma(\omega(Z,X_{2k-1}),X_{2k}) -
         \sigma(\omega(Z,X_{2k}), X_{2k-1})\\
       &=\sum_{j=1}^s g(\omega(Z,Y_{2j-1}),Y_{2j-1}) +
         g(\omega(Z,Y_{2j}),Y_{2j}) \\
       &\qquad
         + \sum_{k=1}^r g(\omega(Z,X_{2k-1}),X_{2k-1}) +
         g(\omega(Z,X_{2k}),X_{2k})\\
       &=\tr(\omega(Z,\any))= - \tr(\ad(Z)).
    \end{split}
  \end{equation*}

  For the next case, \( W = X_{2r} \) is omitted and we put
  \( X = X_{2r-1} \).
  As before, the only contributions are from inserting \( X \) and two
  \( J \)-linearly dependent vectors into
  \( \sigma(\omega(\any,\any),\any) \).
  Noting that \( \omega|_{\Lambda^2 \mfa} = 0 \) and
  \( \im(\omega) \perp V_{J} \), one computes
  \begin{align*}
    0
    &= \sum_{j=1}^s \sigma(\omega(X,Y_{2j-1}),Y_{2j}) -
      \sigma(\omega(X,Y_{2j}),Y_{2j-1}) \\
    &\qquad
      +\sum_{k=1}^{r-1} \sigma(\omega(X,X_{2k-1}),X_{2k}) -
      \sigma(\omega(X,X_{2k}),X_{2k-1})\\
    &\qquad
      +\sigma\paren[\bigg]{\sum_{k=1}^{r-1} \omega(X_{2k-1},X_{2k})
      + \sum_{k=1}^{\ell} \omega(Z_{2k-1},Z_{2k}),X}\\
    &= \sum_{p=1}^{2s} g(\omega(X,Y_p),Y_p)
      +\sum_{q=1}^{2r-2} g(\omega(X,X_q),X_q)
      +\sigma(-C + [X,W], X) \\
    &= -\tr(\ad(X)) + g([X,W],W)
      - \sigma\paren{C,X}+g(J[X,W],X)\\
    & = -\tr(\ad(X)) - \sigma(C,X).
  \end{align*}

  Finally, for \( W = Y_{2s} \), \( Y = Y_{2s-1} \), we need to have
  \( Y \) as the final argument of
  \( \sigma(\omega(\any,\any),\any) \), and the first two arguments
  need to be \( J \)-dependent vectors from \( V_{r}\oplus V_{J} \).
  So \eqref{eq:balancedsheardata} on these vectors reduces to
  \begin{equation*}
    0= -\sigma(C,Y) = g(C,JY),
  \end{equation*}
  giving \( C \) is orthogonal to~\( \derg_{J} \).
\end{proof}

\begin{remark}
  \label{rem:pair-basis}
  The proof shows that the conditions of
  Proposition~\ref{pro:balanced} hold for one pair
  \( X_1,\dots,X_{2r} \) and \( Z_1,\dots,Z_{2\ell} \) of unitary
  bases, then they hold automatically for all such pairs.
\end{remark}

\begin{remark}
  Proposition~\ref{pro:balanced} allows us to recover the
  classification of certain balanced Hermitian Lie algebras
  in~\cite{FiP1}.
  For this, let \( (\mfg, g, J) \) be a \( 2n \)-dimensional almost
  Abelian Hermitian Lie algebra, meaning that there is a
  codimension~\( 1 \) Abelian ideal.
  For simplicity, let us just consider the generic case when
  \( \dim(\derg)=2n-1 \).
  With more effort the methods also apply to \( \dim(\derg) < 2n-1 \).
  Under this additional assumption, \( (\mfg,g,J) \) is of pure
  type~III with \( \dim(\derg_J) = 2n-2 \) and
  \( \dim(\derg_r) = 1 \).
  Choose \( X\in \derg_r \) of norm one and consider
  \( A_X = \omega(JX,\any) \in \End(\derg) \).
  By Lemma~\ref{le:conscomplexsheardata}, there exist \( a\in \bR \),
  \( v\in \derg_J \) and \( A\in \gl(\derg_J,J) \) such that
  \begin{equation*}
    A_X=
    \begin{pmatrix}
      a & v \\
      0 & A
    \end{pmatrix}
  \end{equation*}
  with respect to the splitting \( \derg=\derg_r\oplus \derg_J \).
  By Proposition~\ref{pro:balanced}, \( (\mfg,g,J) \) is balanced if
  and only if
  \begin{equation*}
    a+\tr(A)=\tr(\ad(JX))=-\sigma([X,JX],JX)=g(aX,X)=a
  \end{equation*}
  and \( [X,JX] = v \) is orthogonal to \( \derg_{J} \).
  But \( v \in \derg_{J} \), so we have balanced if and only if
  \( \tr(A) = 0 \) and \( v = 0 \), which coincides with~\cite{FiP1}.
\end{remark}

For unimodular Lie algebras \( \tr(\ad(X)) = 0 \) for all
\( X \in \mfg \), so Proposition~\ref{pro:balanced} simplifies as
below and Remark~\ref{rem:pair-basis} holds in this context.

\begin{corollary}
  \label{co:balancedunimodular}
  Let \( (\mfg,g,J) \) be a unimodular Hermitian two-step solvable Lie
  algebra.
  Then \( (g,J) \) is balanced if and only if \( C = 0 \) in
  \eqref{eq:C}.  \qed
\end{corollary}

Note that in Proposition~\ref{pro:balanced}, the restriction of
\( g \) to \( \derg_J \) plays no role.

\begin{corollary}
  \label{co:changebalancedmetric}
  Let \( (\mfg,g,J) \) be a two-step solvable balanced Lie algebra.
  If \( \tilde{g} \) is another metric compatible with \( (\mfg,J) \)
  satisfying \( (\derg_J)^{\perp g} = (\derg_J)^{\perp \tilde{g}} \)
  and
  \( g|_{(\derg_J)^{\perp g}} = \tilde{g}|_{(\derg_J)^{\perp
  \tilde{g}}} \), then \( \tilde{g} \) is balanced too.  \qed
\end{corollary}

\subsection{K\"ahler Lie algebras}
\label{subsec:Kahler}

\begin{lemma}
  \label{le:Kahlersheardata}
  Let \( (\bR^{2n},g,J) \) be a flat K\"ahler structure on the Lie
  algebra \( \bR^{2n} \) with associated K\"ahler form \( \sigma \).
  Then the shear \( (\mfg,g_{\mfg},J_{\mfg}) \) of
  \( (\bR^{2n},g,J) \) by complex shear data \( (\mfa,\omega) \) is
  K\"ahler if and only if
  \begin{equation}
    \label{eq:Kahlersheardata}
    \cA(\sigma(\omega(\any,\any),\any))=0.
  \end{equation}
\end{lemma}

\begin{proof}
  Similar to the proof of Lemma~\ref{le:balancedsheardata},
  \cite[Corollary 3.11]{FrSw1} implies
  \begin{equation*}
    \begin{split}
      d_{\mfg} \sigma_{\mfg}
      &= -\cA(\sigma(\omega(\any,\any),\any))
    \end{split}
  \end{equation*}
  and so that \( (\mfg,g_{\mfg},J_{\mfg}) \) is K\"ahler, i.e.\
  \( d_{\mfg}\sigma_{\mfg}=0 \), if and only
  if~\eqref{eq:Kahlersheardata} holds.
\end{proof}

We first note some general consequences of~\eqref{eq:Kahlersheardata}.

\begin{lemma}
  \label{le:Kahlersheardataconsequences}
  When the shear in Lemma~\ref{le:Kahlersheardata} is K\"ahler, we
  have
  \begin{equation*}
    \omega(U_J,U_J)=0, \qquad \omega(J\mfa_r,J\mfa_r)=0
    \quad\text{and}\quad h=0.
  \end{equation*}
  Furthermore, there exist a complex unitary basis \( Y_1,\dots,Y_s \)
  of \( \mfa_J \), a real orthonormal basis \( X_1,\dots,X_r \) of
  \( \mfa_r \), one-forms \( \alpha_1,\dots,\alpha_s\in \mfa_r^* \) on
  \( \mfa_r \), and \( \lambda_1,\dots,\lambda_r\in \bR \) such that
  \begin{equation*}
    K_X(Y_j) = - \alpha_j(X) JY_j,\qquad
    f(X_k,X_m)= - \delta_{km} \lambda_k X_{k}
  \end{equation*}
  for all \( X\in \mfa_r \), \( Z\in U_J \), \( j\in \{1,\dots,s\} \)
  and \( k,m\in \{1,\dots,r\} \).
\end{lemma}

\begin{proof}
  Inserting \( Z_1,Z_2\in U_J \), which is orthogonal to~\( \mfa \),
  and \( W\in \mfa+J\mfa \) into~\eqref{eq:Kahlersheardata} yields
  \begin{equation*}
    0 = \sigma(\omega(Z_1,Z_2),W),
  \end{equation*}
  and so \( \omega(Z_1,Z_2)=0 \), giving \( \omega(U_J,U_J)=0 \).
  Next, observe that for \( \hat{X},\tilde{X}\in \mfa_r \), the
  endomorphisms \( K_{\tilde{X}}, K_{\hat{X}}\in \End(\mfa_J) \) are
  complex and commute by Lemma~\ref{le:conscomplexsheardata}.
  Inserting \( \tilde{Y},\hat{Y}\in\mfa_J \) and
  \( JX \in J\mfa_{r} \) into~\eqref{eq:Kahlersheardata} yields
  \begin{equation*}
    0 = \sigma(K_X(\tilde{Y}),\hat{Y})+\sigma(\tilde{Y},K_X(\hat{Y})),
  \end{equation*}
  meaning that \( K_X\in \sP(\mfa_J,\sigma) \).
  But \( K_{X} \) is complex, so \( K_X\in \un(\mfa_J,g,J) \).
  It follows, that the \( K_X \) are simultaneously complex
  diagonalisable with imaginary eigenvalues.
  In particular, there exists a complex unitary basis
  \( Y_1,\dots,Y_s \) of \( (\mfa_J,g) \) and one-forms
  \( \alpha_1,\dots,\alpha_s\in \mfa_r^* \) on \( \mfa_r \) such that
  \begin{equation*}
    K_X(Y_j) = - \alpha_j(X) JY_j
  \end{equation*}
  for all \( X\in \mfa \) and all \( j\in \{1,\dots,s\} \).
  Next, \eqref{eq:Kahlersheardata} evaluated on \( Y\in \mfa_J \),
  \( \tilde{X} \in \mfa_r \) and \( J\hat{X} \in J\mfa_{r} \) shows
  \begin{equation*}
    0 = \sigma(Y,H_{\hat{X}}(\tilde{X}))
    = \sigma(Y,h(\hat{X},\tilde{X})),
  \end{equation*}
  so \( h=0 \).
  Lemma~\ref{le:conscomplexsheardata} gives
  \( \omega(J\tilde{X},J\hat{X}) =
  J(h(\tilde{X},\hat{X})-h(\hat{X},\tilde{X})) = 0 \), and hence
  \( \omega(J\mfa_r,J\mfa_r)=0 \).
  Finally, putting \( \bar{X},J\tilde{X},J\hat{X} \in \mfa_{r} \)
  in~\eqref{eq:Kahlersheardata} gives
  \begin{equation*}
    0
    = -\sigma(f(\tilde{X},\bar{X}),J\hat{X}) +
    \sigma(f(\hat{X},\bar{X}),J\tilde{X})
    =-g(f(\tilde{X},\bar{X}),\hat{X}) +
    g(f(\hat{X},\bar{X}),\tilde{X}).
  \end{equation*}
  However \( f \) is symmetric by Lemma~\ref{le:conscomplexsheardata},
  we get that \( g(f(\any,\any),\any) \) is totally symmetric.
  In particular, all endomorphisms \( F_X \) are symmetric.
  Since these endomorphisms commute by
  Lemma~\ref{le:conscomplexsheardata}, we have a common orthonormal
  basis \( X_1,\dots,X_r \) of \( (\mfa_r,g) \) of eigenvectors for
  all~\( F_X \).  Now
  \begin{equation*}
    \spa{X_j}\ni F_{X_i}(X_j)=f(X_i,X_j)=f(X_j,X_i)
    =F_{X_j}(X_i)\in \spa{X_i}
  \end{equation*}
  for all \( i,j\in \{1,\dots,r\} \) implies the existence of
  \( \lambda_1,\dots,\lambda_r\in \bR \) such that
  \begin{equation*}
    f(X_i,X_j)= - \delta_{ij} \lambda_j X_j
  \end{equation*}
  for all \( i,j\in \{1,\dots,r\} \).
\end{proof}

It seems hard to solve~\eqref{eq:Kahlersheardata} in full generality,
so we now restrict to a certain subclass, namely those two-step
solvable K\"ahler Lie algebras \( (\mfg,g,J) \) with
\( [J\derg, \derg_J] = \derg_J \).
Note that this class includes those pure type~I.
Moreover, it also includes algebras of pure type~III, i.e.\ with
\( V_J = 0 \), since by Lemma~\ref{le:Kahlersheardataconsequences} we
have \( [J\derg_r, J\derg_r] = 0 \) and \( h = 0 \) and so must have
\( [J\derg, \derg_J] = \derg_J \) in order for \( \derg \) to be the
commutator ideal.

When \( [J\derg, \derg_J] = \derg_J \), the form of \( B_Z \),
\( Z\in V_J \), simplifies and allows for a classification.
The condition \( [J\derg, \derg_J] = \derg_J \) is equivalent to
\( \alpha_j\neq 0 \) for all \( j=1,\dots,r \).
Moreover, for any \( Z \in U_J \), the endomorphisms \( B_Z \) commute
with \( A_X \) for all \( X\in \mfa_r \) by
Lemma~\ref{le:conscomplexsheardata}.
As \( A_X \) has imaginary non-zero eigenvalues on \( \mfa_J \) and
real eigenvalues on \( \mfa_r \), we get that \( B_Z \) preserves the
splitting \( \mfa=\mfa_r\oplus \mfa_J \).
Using this property, we will obtain:

\begin{theorem}
  \label{th:KahlerdergdergJ=dergJ}
  Let \( (\mfg,g,J) \) be an almost Hermitian Lie algebra.
  Then \( (\mfg,g,J) \) is a two-step solvable K\"ahler Lie algebra
  \( (\mfg,g,J) \) with \( [J\derg,\derg_J]=\derg_J \) if and only if
  \( J\derg_r\perp \derg_r \) and there exist a complex unitary basis
  \( Y_1,\dots, Y_s \) of~\( \derg_J \), an orthonormal basis
  \( X_1,\dots,X_r \) of~\( \derg_r \), non-zero one forms
  \( \alpha_1,\dots,\alpha_s\in (\derg_r)^*\setminus \{0\} \),
  one-forms \( \beta_1,\dots, \beta_s\in V_J^* \) and non-zero real
  numbers \( \lambda_1,\dots,\lambda_r\in \bR\setminus \{0\} \) such
  that the only non-zero Lie brackets (up to anti-symmetry and complex
  linear extension on \( \derg_J \)) are given by
  \begin{equation*}
    [JX,Y_j]=\alpha_j(X) JY_j,\qquad
    [Z,Y_j]=\beta_j(Z) JY_j,\qquad [JX_k,X_k]=\lambda_k X_k
  \end{equation*}
  for \( j\in \{1,\dots,s\} \), \( k\in \{1,\dots,r\} \),
  \( X\in \derg_r \) and \( Z\in V_J \).
\end{theorem}

For the proof, we first need the following result.

\begin{lemma}
  \label{le:symplecticendosareunitary}
  Let \( V \) be \( 2n \)-dimensional vector space endowed with a
  Hermitian structure \( (g,J) \) and denote by \( \sigma \) the
  associated fundamental two-form.
  Suppose \( A_1,A_2\in \sP(V,\sigma) \) satisfy \( [A_1,A_2]=0 \) and
  \( A_1+JA_1J+JA_2-A_2 J=0 \).
  Then we have \( A_1,A_2\in \un(V,g,J) \).
\end{lemma}

\begin{proof}
  For \( i=1,2 \), decompose \( A_i=A_i^J+A_i^{J-} \) into the sum of
  its \( J \)-invariant part~\( A_i^J \) and its
  \( J \)-anti-invariant part~\( A_i^{J-} \).
  Then \( A_i^J \in \un(V,g,J) \), and \( A_i^{J-} \) is symmetric
  with respect to \( g \).

  Now
  \begin{equation*}
    0 = A_1+JA_1J+JA_2-A_2 J
    = J [A_1,J]-[A_2,J]
    = 2 J A_1^{J-}-2 A_2^{J-},
  \end{equation*}
  so
  \begin{equation*}
    A_2^{J-} = JA_1^{J-}.
  \end{equation*}
  Moreover, the \( J \)-invariant part of \( 0=[A_1,A_2] \) yields
  \begin{equation*}
    0 = [A_1^J,A_2^J] + [A_1^{J-},A_2^{J-}]
    = [A_1^J,A_2^J] - 2 J A_1^{J-} A_1^{J-}.
  \end{equation*}
  Since \( A_1^J\in \un(V,g,J) \), there exists a basis of \( V \)
  consisting of vectors \( v\in V \) with \( A_1^J v=c v \) and
  \( A_1^{J-} Jv=-cv \) for some \( c\in \bR \).
  For such a vector \( v \), we obtain
  \begin{equation*}
    \begin{split}
      \MoveEqLeft
      g(Jv,[A_1^J,A_2^J]v) \\
      &= g(Jv, A_1^J A_2^J v)-g(Jv, A_2^J A_1^J v)
        = -g(A_1^J Jv, A_2^J v)-c\, g(Jv, A_2^J Jv)\\
      &= c\, g(v, A_2^J v)-c\, g(Jv, J A_2^J v)
        = c\, g(v, A_2^J v)-c\, g(v, A_2^J v)=0.
    \end{split}
  \end{equation*}
  Hence,
  \begin{equation*}
    0=g(Jv,J A_1^{J-} A_1^{J-}v)=g(v, A_1^{J-}
    A_1^{J-}v)=g(A_1^{J-}v,A_1^{J-}v) = \norm{A_1^{J-}v}^2.
  \end{equation*}
  This shows \( A_1^{J-}=0 \) and so also \( A_2^{J-}=J A_1^{J-}=0 \),
  finishing the proof.
\end{proof}

\begin{lemma}
  \label{lem:split-B}
  If the shear in Lemma~\ref{le:Kahlersheardata} is K\"ahler and for
  each \( Z \in U_{J} \), \( B_{Z} \) preserves the splitting
  \( \mfa = \mfa_{J} + \mfa_{r} \), then the \( Y_{j} \) in
  Lemma~\ref{le:Kahlersheardataconsequences} may be chosen so that
  \( B_{Z}(Y_{j}) = - \beta_{j}(Z)JY_{j} \), for all \( j \), for some
  \( \beta_{j} \in U_{J}^{*} \).
\end{lemma}

\begin{proof}
  Write \( B_Z = b_Z+c_Z \) with \( b_Z\in \End(\mfa_J) \) and
  \( c_Z\in \End(\mfa_r) \).
  Then \( [b_Z,b_{JZ}]=0 \) and \( [b_Z,K_X]=0 \) for all
  \( X\in \mfa_r \).
  Moreover, the second equation in~\eqref{eq:complexsheardata} yields
  \begin{equation*}
    \begin{split}
      b_{JZ} J(Y)
      &=J^*\omega(JZ,JY)=\omega(Z,Y)+J(\omega(JZ,Y)+\omega(Z,JY))\\
      &=b_Z(Y)+J b_{JZ}(Y)+Jb_Z J (Y)
    \end{split}
  \end{equation*}
  for all \( Y\in \mfa_J \), and so
  \begin{equation*}
    b_Z+Jb_Z J+J b_{JZ}-b_{JZ} J=0.
  \end{equation*}
  Inserting \( Z,\tilde{Y},\hat{Y} \), where
  \( \tilde{Y},\hat{Y}\in \mfa_J \), into~\eqref{eq:complexsheardata}
  gives
  \begin{equation*}
    0 = \sigma(b_Z(\tilde{Y}),\hat{Y})
    + \sigma(\tilde{Y},(b_Z(\hat{Y})),
  \end{equation*}
  so \( b_Z\in \sP(\mfa_J,\sigma) \).
  Thus, we deduced from Lemma~\ref{le:symplecticendosareunitary} that
  \( b_Z\in \un(\mfa_J,g,J) \).
  Lemma~\ref{le:conscomplexsheardata} gives that all \( b_Z \) commute
  pairwise and with all \( K_X \), \( X\in \mfa_r \), and the result
  follows.
\end{proof}

\begin{proof}[Proof of Theorem~\ref{th:KahlerdergdergJ=dergJ}]
  Using shear data, let \( Z\in U_J \) be given.
  As we argued before Theorem~\ref{th:KahlerdergdergJ=dergJ},
  \( B_Z \)~preserves the subspaces \( \mfa_J \) and \( \mfa_r \), so
  we may apply Lemma~\ref{lem:split-B} to get the \( Y_{j} \),
  \( \alpha_{j} \) and \( \beta_{j} \).
  The hypothesis \( [J\derg,\derg_{J}] = \derg_{J} \) implies that
  each \( \alpha_{j} \) is non-zero.

  Next, by Lemma~\ref{le:conscomplexsheardata}, we have
  \( \omega^{r}(JZ,JX) = \omega^{r}(Z,X) = c_Z(X) \) for any
  \( Z\in U_J \), \( X\in \mfa_r \).  Hence,
  \begin{equation*}
    \begin{split}
      0
      &= \cA(\sigma(\omega(\any,\any),\any))(JZ, J\tilde{X},
        J\hat{X})
        = \sigma(c_Z(\tilde{X}),J\hat{X}) -
        \sigma(c_Z(\hat{X}),J\tilde{X})\\
      &=g(c_Z(\tilde{X}),\hat{X}) - g(c_Z(\hat{X}),\tilde{X}),
    \end{split}
  \end{equation*}
  and \( c_Z \) is symmetric.
  Since all \( c_Z \) commute pairwise and with all \( F_X \),
  \( X\in \mfa_r \), by Lemma~\ref{le:conscomplexsheardata}, there is
  a common eigenbasis \( X_1,\dots,X_r \) of \( \derg_r \) for all
  these operators.
  Moreover, by Lemma~\ref{le:Kahlersheardataconsequences}, there exist
  \( \lambda_1,\dots,\lambda_r\in \bR \) with
  \begin{equation*}
    f(X_j,X_k) = F_{X_j}(X_k) = - \delta_{jk}\lambda_k X_k.
  \end{equation*}
  Next,
  \begin{equation*}
    \begin{split}
      0
      &= \cA(\sigma(\omega(\any,\any),\any))(Z,X_j,JX_j)
        = \sigma(c_Z(X_j),JX_j) - \sigma(\omega(Z,JX_j),X_j)\\
      &= g(c_Z(X_j),X_j) + \sigma(c_{JZ}(X_j),X_j)
        = g(c_Z(X_j),X_j)
    \end{split}
  \end{equation*}
  for any \( j\in \{1,\dots,r\} \) since
  \( c_{JZ}(X_j)\in \spa{X_j} \).
  As \( c_Z(X_j)\in \spa{X_j} \), we get \( c_Z(X_j)=0 \) and hence
  \( c_Z=0 \).
  So \( \omega^{r}(U_J,J\mfa_r)=0 \) as well.
  Hence, \( f\colon \mfa_r\times \mfa_r\rightarrow \mfa_r \) has to be
  surjective in order to have \( \im(\omega)=\mfa_r \) and thus
  \( \lambda_j\neq 0 \) for all \( j=1,\dots,r \).

  Next, inserting \( X_j,X_k, JX_k \) for \( j,k\in \{1,\dots,r\} \)
  with \( j\neq k \) into~\eqref{eq:Kahlersheardata} yields
  \begin{equation*}
    0 = \cA(\sigma(\omega(\any,\any),\any))(X_j,X_k,JX_k)
    = -\sigma(f(X_k,X_k),X_j)
    = \lambda_k g(JX_k,X_j).
  \end{equation*}
  Since \( \lambda_k\neq 0 \) and trivially \( g(JX_k,X_k)=0 \) holds,
  this implies \( X_k\perp J\mfa_r \), so \( \mfa_r\perp J\mfa_r \).

  We now have all the claimed properties except that
  \( \omega(Z,JX) = 0 \) for \( Z\in U_J \), \( X\in \mfa_r \).
  However, we already showed that \( \omega(Z,JX)\in \mfa_J \).
  Moreover, inserting \( Z,JX,Y \) into~\eqref{le:Kahlersheardata} for
  \( Y\in \mfa_J \), \( X\in \mfa_r \) and \( Z\in U_J \) yields
  \begin{equation*}
    0=\cA(\sigma(\omega(\any,\any),\any))(Z,JX,Y)
    =\sigma(\omega(Z,JX),Y).
  \end{equation*}
  Thus, \( \omega(Z,JX)=0 \) for all \( X\in \mfa_r \),
  \( Z\in U_J \), which completes the proof.
\end{proof}

For the three different pure types we arrive at the classification
below.

\begin{corollary}
  \label{co:Kahlerpuretype}
  Let \( (\mfg,g,J) \) be a \( 2n \)-dimensional almost Hermitian Lie
  algebra that is two-step solvable.  Then we have the following.
  \begin{enumerate}
  \item[\textup{(I)}] \( (\mfg,g,J) \) is K\"ahler of pure type~I if
    and only if \( J\derg \perp \derg \), and there exists an
    orthonormal basis \( X_1,\dots,X_r \) of \( \derg \) and
    \( \lambda_1,\dots,\lambda_r\in \bR\setminus \{0\} \) such that
    the only non-zero Lie brackets (up to anti-symmetry) are given by
    \begin{equation*}
      [JX_j,X_j]=\lambda_j X_j
    \end{equation*}
    for \( j=1,\dots,r \).
  \item[\textup{(II)}] \( (\mfg,g,J) \) is K\"ahler of pure type~II if
    and only if there exists a complex unitary basis
    \( Y_1,\dots,Y_s \) of \( \derg \) and non-zero one-forms
    \( \beta_1,\dots,\beta_s\in V_J^* \) such that the only non-zero
    Lie brackets (up to anti-symmetry and complex-linear extension)
    are given by
    \begin{equation*}
      [Z,Y_j]=\beta_j(Z) JY_j
    \end{equation*}
    for \( j\in \{1,\dots,s\} \) and \( Z\in V_J \).
  \item[\textup{(III)}] \( (\mfg,g,J) \) is K\"ahler of pure type~III
    if and only if \( J\derg_r\perp \derg_r \), and there exist a
    complex unitary basis \( Y_1,\dots,Y_s \) of \( \derg_J \), an
    orthonormal basis \( X_1,\dots,X_r \) of \( \derg_r \), non-zero
    one-forms
    \( \alpha_1,\dots,\alpha_s\in (\derg_r)^*\setminus \{0\} \) on
    \( \derg_r \) and non-zero real numbers
    \( \lambda_1,\dots,\lambda_k\in \bR\setminus \{0\} \) such that
    the only non-zero Lie brackets (up to anti-symmetry and
    complex-linear extension) are given by
    \begin{equation*}
      [JX_{k},Y_j] = \alpha_j(X_{k}) JY_j,\qquad
      [JX_k,X_k]=\lambda_k X_k.
    \end{equation*}
  \end{enumerate}
\end{corollary}

\begin{proof}
  For pure types I and III this is just specialisation of
  Theorem~\ref{th:KahlerdergdergJ=dergJ}.
  For pure type~II, we have \( \derg_{r} = 0 \), so we may use
  Lemma~\ref{lem:split-B}.
\end{proof}

\begin{remark}
  For pure type~I this implies
  \( (\mfg,J) \cong r(\aff_{\bR},J)\oplus (\bR^{2(n-r)},J) \) as Lie
  algebras with complex structures.
  Up to change of basis, complex structures on \( \aff_{\bR} \) and
  \( \bR^{2(n-r)} \) are unique.
\end{remark}

\section{Compatibility of balanced and SKT metrics}
\label{sec:compatibility}

We now consider the question of Lie groups or Lie algebras with a
complex structure that admit both a compatible balanced metric
\( \hat{g} \) and a compatible SKT metric \( \tilde{g} \).
We will say that a complex structure \( J \) is \emph{SKT},
\emph{balanced} or \emph{K\"ahler} if it admits a compatible metric
that is SKT, balanced or K\"ahler, respectively.

We first consider some general results for compact groups and for
solvable groups in \S\ref{sec:general-results}.
Thereafter, we will just consider two-step solvable Lie algebras,
considering each pure type in turn, and then specialising to the
six-dimensional case.

\subsection{General results}
\label{sec:general-results}

\begin{theorem}
  \label{thm:compact}
  Let \( J \) be a left-invariant complex structure on a compact
  group~\( G \).
  Then \( J \) is SKT, but is only balanced if \( G \) is Abelian, in
  which case it is also K\"ahler.
\end{theorem}

The special case when \( G \) is also semi-simple was proved in
\cite{FGV,P}.

\begin{proof}
  Existence of the SKT metric was given in~\cite{SSTvP,MaSw}.

  Now suppose that \( G \) is not Abelian.
  It is sufficient to assume \( G \) is connected.
  Note that \( \mfg = \bR^{k} \oplus \mfk \) for some semi-simple Lie
  algebra \( \mfk \) of compact type.
  It follows that there is a finite cover of \( G \) by the group
  \( T^{k} \times K \), where \( K \)~is compact, connected and simply
  connected with Lie algebra~\( \mfk \), and it is sufficient to
  consider the case \( G = T^{k}\times K \).
  By \cite[Theorem in~(3.2)]{Sn} there is a Cartan subgroup \( N \)
  of~\( G \) and holomorphic fibration \( \pi\colon G \to G/N \) with
  \( G/N \) projective.
  Note that \( N = T^{k} \times N_{1} \), with \( N_{1} \)~a Cartan
  subgroup of~\( K \).
  Let \( Y \) be a complex submanifold of \( G/N \) of complex
  codimension one, and put \( X = \pi^{-1}(Y) \).
  Then \( X = T^{k} \times (X\cap K) \) and is of real codimension
  two.
  In particular, for the fundamental class \( [X] \) of \( X \) and
  any generator \( c \in H^{m}(K,\bZ) \cong \bZ \), we have
  \( c \frown [X] = 0 \).

  On the other hand, the Whitehead Theorems imply that
  \( H^{1}(K) = 0 = H^{2}(K) \), so by duality
  \( H^{m-1}(K) = 0 = H^{m-2}(K) \), where \( m = \dim K \).
  Writing \( 2n = k+m = \dim G \), the K\"unneth formula gives
  \( H^{2n-2}(G) = \bigoplus_{s=0}^{2} H^{k-2+s}(T^{k}) \otimes
  H^{m-s}(K) = H^{k-2}(T^{k}) \otimes H^{m}(K) \).
  If \( \sigma \) is the two-form of a balanced metric, we have
  \( d(\sigma^{n-1}) = 0 \), so
  \( [\sigma^{n-1}] = b \otimes_{\bR} c \) for some
  \( b \in H^{k-2}(T^{k}) \).
  But now
  \( 0 < \int_{X}\sigma^{n-1} = (b\otimes c) \frown [X] = b\frown
  (c\frown [X]) = 0 \), which is a contradiction.
  Thus if \( J \) is balanced, then \( K = \Set{e} \) and \( G \) is
  Abelian.
  But then \( J \) is a left-invariant complex structure on a torus
  and so admits a compatible K\"ahler metric.
\end{proof}

Now let us show that invariance of the K\"ahler metric is necessary in
Question~\ref{quest:G}.

\begin{proposition}
  \label{pro:Kahleronsolvable}
  Let \( G \) be a simply-connected solvable Lie group and let \( J \)
  be a left-invariant complex structure on \( G \).
  Then \( (G,J) \) admits a compatible K\"ahler metric.
\end{proposition}

\begin{proof}
  By the theorem in~\cite[(1.3)]{Sn}, there is a discrete subgroup
  \( \Gamma \) of \( G \) such that \( \Gamma\backslash G \) is
  biholomorphic to an open subset \( V \) of \( \bC^n \).
  Now \( \bC^n \), and hence \( V \), carries a compatible K\"ahler
  metric that we may pull back under the natural projection
  \( \pi\colon G\rightarrow \Gamma\backslash G\cong V\subseteq \bC^n
  \) to get a compatible K\"ahler metric on~\( (G,J) \).
\end{proof}

\subsection{Pure type~I}
\label{subsec:puretypeI}

Pure type~I gives \( \derg_{J} = 0 \), so \( \derg \) is totally real.

\begin{theorem}
  \label{th:compatibilitypuretypeI}
  Let \( (\mfg,J) \) be a unimodular two-step solvable Lie algebra
  \( \mfg \) with complex structure \( J \) of pure type~I that is SKT
  and is balanced.
  Then \( \mfg \) is Abelian and so \( J \) is K\"ahler.
\end{theorem}

\begin{proof}[Proof of Theorem~\ref{th:compatibilitypuretypeI}]
  Without unimodularity, the structure of the SKT algebras is given
  in~\cite[Theorem 5.5]{FrSw2}:
  \( \mfg \cong m\aff_{\bR} \oplus \mfh \), for some nilpotent Lie
  algebra~\( \mfh \).
  But \( \mfh \) is unimodular and \( \aff_{\bR} \) is not, so
  \( \mfg \) is unimodular if and only if \( m = 0 \).
  \cite[Theorem~5.5]{FrSw2} now gives that \( \mfg = \mfh \) is
  two-step nilpotent.
  As \( J \) is SKT, is balanced and \( \mfg \) is two-step nilpotent,
  \cite[proof of Theorem~1.1]{FiV2} shows that \( (\mfg,J) \) is
  K\"ahler.
\end{proof}

Next, we provide an example that shows that the unimodular condition
in Theorem~\ref{th:compatibilitypuretypeI} is necessary, and hence is
also necessary in Question~\ref{quest:G}.

\begin{example}
  \label{ex:counterexampletypeI}
  Let \( \mfg = \aff_{\bR} \oplus \mfh_3 \oplus \bR \).
  Then \( \mfg \) is a non-unimodular two-step solvable Lie algebra
  with a basis \( e_1,\dots,e_6 \) for which (up to anti-symmetry) the
  only non-zero Lie brackets are
  \begin{equation*}
    [e_1,e_2]=e_2 \quad\text{and}\quad [e_3,e_4]=e_5.
  \end{equation*}
  Let \( J \) be the almost complex structure with
  \( Je_{2i-1}=e_{2i} \) for \( i=1,2,3 \).
  Thus \( (\mfg, J) \) is a direct sum of \( (\aff_{\bR},J_{1}) \) and
  \( (\mfh_{3}\oplus \bR, J_{2}) \).
  For the dual basis \( e^{1},\dots,e^{6} \) the non-zero
  differentials are \( de^{2} = -e^{12} \) and \( de^{5} = -e^{34} \).
  Then \( J_{1} \) is integrable, and for \( J_{2} \) the
  \( (1,0) \)-forms are spanned by \( \alpha_{1} = e^{3} - ie^{4} \)
  and \( \alpha_{2} = e^{5} - ie^{6} \).
  As \( d\alpha_{1} = 0 \) and
  \( d\alpha_{2} = - e^{3}\wedge e^{4} = \frac{i}{2} \alpha_{1} \wedge
  \overline{\alpha_{1}} \) is of type \( (1,1) \), we have that
  \( J_{2} \), and hence \( J \), is integrable.
  Moreover, \( \derg=\spa{e_2,e_5} \) is totally real, so
  \( (\mfg,J) \) is of type~I.

  Let \( \tilde{g} \) be the metric on \( \mfg \) for which
  \( e_1,\dots,e_6 \) is an orthonormal basis.
  Then \( \tilde{g} \) is compatible with \( J \), and the associated
  fundamental two-form is
  \begin{equation*}
    \tilde{\sigma}=e^{12}+e^{34}+e^{56}.
  \end{equation*}
  We now find that
  \begin{equation*}
    dJ^*d\tilde{\sigma} = -dJ^{*} e^{346} = d e^{345} =0,
  \end{equation*}
  so \( \tilde{g} \) is an SKT metric.

  Next, let \( \hat{g} \) be the metric on \( \mfg \) for which
  \( e_1-e_6,e_2+e_5,e_6,-e_5,e_3,e_4 \) is an orthonormal basis.
  Since this basis is unitary, \( \hat{g} \) is compatible
  with~\( J \).
  Moreover, a dual basis is given by
  \( e^1,e^2,e^1+e^6,e^2-e^5,e^3,e^4 \) and so the associated
  fundamental form \( \hat{\sigma} \) is given by
  \begin{equation*}
    \hat{\sigma}=e^{12}+(e^1+e^6)\wedge (e^2-e^5)+e^{34}
    =2e^{12}-e^{15}-e^{26}+e^{34}+e^{56}.
  \end{equation*}
  We now have
  \begin{equation*}
    \begin{split}
      d(\hat{\sigma}^2)
      &= 2 \hat{\sigma} \wedge d\hat{\sigma}
        = 2\hat{\sigma} \wedge (-e^{134} + e^{126} - e^{346}) \\
      &= 2(e^{12346} - e^{13456} + e^{12346} - 2e^{12346} + e^{13456})
        = 0,
    \end{split}
  \end{equation*}
  and hence that \( \hat{g} \) is balanced.

  However, \( \mfg\not\cong r\aff_{\bR}\oplus \bR^{6-2r} \) for any
  \( r\in \{1,\dots,3\} \), so by
  Corollary~\ref{co:Kahlerpuretype}(a), \( (\mfg,J) \) does not admit
  any compatible K\"ahler metric.
\end{example}

\subsection{Pure type~II}
\label{subsec:puretypeII}

Pure type~II means that \( \derg \) is complex.
We will first classify two-step solvable SKT Lie algebras
\( (\mfg,g,J) \) of pure type~II up to some remaining ``nilpotent''
equations and give a full classification if \( \derg \) is of
codimension two.
The latter case was the remaining open case in our classification of
the six-dimensional two-step solvable Lie algebras admitting an SKT
structure in~\cite[Theorem 7.1]{FrSw2} and so we complete this
classification here in Theorem~\ref{th:classification6dSKT}.

We begin by deriving some consequences for the form of the
endomorphisms~\( B_Z \).

\begin{lemma}
  \label{le:SKTpuretypeII}
  Let \( (\mfg,g,J) \) be a two-step solvable SKT Lie algebra of pure
  type~II.
  Then there exists a complex unitary basis \( Y_1,\dots,Y_r \) of
  \( \derg=\derg_J \) and one-forms
  \( \alpha_1,\dots,\alpha_r\in V_J^* \) such that for any
  \( Z\in V_J \), the endomorphism
  \( \ad(Z)|_{\derg}\in \End(\derg_J) \) is complex and satisfies
  \begin{equation*}
    [Z,Y_j] = \alpha_j(Z) JY_j
  \end{equation*}
  for all \( j\in \{1,\dots,r\} \).
\end{lemma}

\begin{proof}
  Working with shear data, choose \( Z\in U_J \) and set
  \( B_1 \coloneqq B_Z \), \( B_2 \coloneqq B_{JZ} \).
  Then \( [B_1,B_2]=0 \) by the first equation
  in~\eqref{eq:complexsheardata}.
  Moreover, the second equation in~\eqref{eq:complexsheardata} yields
  \begin{equation*}
    \begin{split}
      B_2J(Y)
      &= \omega(JZ,JY)
        = \omega(Z,Y) + J(\omega(Z,JY) + \omega(JZ,Y)) \\
      &= B_1(Y) + JB_1J(Y) + JB_2(Y),
    \end{split}
  \end{equation*}
  for all \( Y\in \mfa \), so
  \begin{equation*}
    B_1+JB_1J+JB_2-B_2J=0.
  \end{equation*}
  Next, inserting \( \tilde{Y},\hat{Y}\in\mfa \) and \( Z,JZ \)
  into~\eqref{eq:SKTsheardata} yields
  \begin{equation*}
    \begin{split}
      0
      &= -g(\omega(JZ,J\tilde{Y}),
        \omega(JZ,\hat{Y}))+g(\omega(JZ,J\hat{Y}),
        \omega(JZ,\tilde{Y})) -g(\omega(Z,J\tilde{Y}),
        \omega(Z,\hat{Y})) \\
      &\qquad
        +g(\omega(Z,J\hat{Y}),
        \omega(Z,\tilde{Y})) +
        g(\omega(J\omega(Z,\tilde{Y}),Z),\hat{Y}) -
        g(\omega(J\omega(Z,\hat{Y}),Z),\tilde{Y}) \\
      &\qquad
        + g(\omega(J\omega(JZ,\tilde{Y}),JZ),\hat{Y}) -
        g(\omega(J\omega(JZ,\hat{Y}),JZ),\tilde{Y})\\
      &= -g(B_2 J(\tilde{Y}),B_2(\hat{Y}))
        +g(B_2 J(\hat{Y}),B_2(\tilde{Y}))
        -g(B_1 J(\tilde{Y}),B_1(\hat{Y})) \\
      &\qquad
        +g(B_1 J(\hat{Y}),B_1(\tilde{Y}))
        -g(B_1 J B_1(\tilde{Y}),\hat{Y})
        +g(B_1 J B_1(\hat{Y}),\tilde{Y}) \\
      &\qquad
        -g(B_2 J B_2(\tilde{Y}),\hat{Y})
        +g(B_2 J B_2(\hat{Y}),\tilde{Y})\\
      &= -g\paren[\Big]{\paren[\big]{ \cC(B_{1})+\cC(B_{2})}\tilde{Y},
        \hat{Y}}
        +g\paren[\Big]{\tilde{Y}, \paren[\big]{\cC(B_{1})+\cC(B_{2})}
        \hat{Y}},
    \end{split}
  \end{equation*}
  where \( \cC(B) \coloneqq B^T B J + B JB \).
  We conclude that \( \cC(B_{1})+\cC(B_{2}) \) is required to be
  \( g \)-symmetric.

  For \( i=1,2 \), decompose \( B = B^J+B^{J-} \) into its
  \( J \)-invariant part \( B^J \) and into its \( J \)-anti-invariant
  part \( B^{J-} \) and then for \( A\in \{J,J-\} \) decompose
  \( B^A \coloneqq B_{+}^A +B_{-}^A \) into its \( g \)-symmetric part
  \( B_{+}^{A} \) and its \( g \)-skew-symmetric part \( B_{-}^A \).
  Then the \( g \)-skew-symmetric part of \( \cC(B) \) is
  \begin{equation}
    \label{eq:skew-part}
    \begin{split}
      \frac{1}{2}(\cC(B) - \cC(B)^{T})
      &= \frac{1}{2}\paren{ B^{T}BJ + BJB + JB^{T}B + B^{T}JB^{T}} \\
      &= \frac{1}{2}\paren{ B^{T}(BJ+JB^{T}) + (BJ+JB^{T})B } \\
      &= B^{T}J(B_{+}^{J}- B_{-}^{J-}) + J(B_{+}^{J}- B_{-}^{J-})B \\
      &= J \paren[\big]{2(B_{+}^{J})^{2}- 2(B_{-}^{J-})^{2}
        + [B_{+}^{J} - B_{-}^{J-},B_{-}^{J} + B_{+}^{J-}]}.
    \end{split}
  \end{equation}
  Note that this has trace
  \begin{equation*}
    2 \tr((B_{+}^{J})^{2}) - 2 \tr((B_{-}^{J-})^{2})=
    2(\norm{B_{+}^{J}}^{2} + \norm{B_{-}^{J-}}^{2}),
  \end{equation*}
  since for \( B^{T} = \varepsilon B \), \( \varepsilon = \pm1 \), and
  any orthonormal basis \( E_1,\dots,E_{2s} \) of \( \mfa \), we have
  \begin{equation*}
    \begin{split}
      \tr(B^{2}) = \sum_{j=1}^{2s}
      g(B^{2}E_{j},E_{j}) = \varepsilon
      \sum_{j=1}^{2s} g(BE_{j},BE_{j}) =
      \varepsilon \norm{B}^{2}.
    \end{split}
  \end{equation*}
  Thus for \( \cC(B_{1})+\cC(B_{2}) \) to be \( g \)-symmetric we must
  have \( B_{i,+}^{J} = 0 = B_{i,-}^{J-} \) for \( i=1,2 \).
  But then \( B_i = B_{i,-}^{J} + B_{i,+}^{J-} \) which lies in
  \( \sP(\mfa,\sigma) \).
  Consequently, we may apply Lemma~\ref{le:symplecticendosareunitary}
  to deduce that we actually have \( B_i\in \un(\mfa,g,J) \), so
  \( B_{i,+}^{J-} = 0 \).
  Then \eqref{eq:skew-part} gives that \( \cC(B_{1}) + \cC(B_{2}) \)
  is \( g \)-symmetric.

  Thus \( B_Z\in \un(\mfa,g,J) \) for \( Z\in U_J \).
  But by Lemma~\ref{le:conscomplexsheardata} gives that all \( B_Z \)
  commute pairwise, so these complex endomorphisms of~\( (\mfa,J) \)
  are simultaneously diagonalisable with only imaginary eigenvalues.
  This is the assertion of Lemma~\ref{le:SKTpuretypeII}
\end{proof}

In the case that \( \derg \) is of codimension two, there are no
further conditions to be satisfied, cf.~\cite[\S7.2]{FrSw2}.  Hence

\begin{corollary}
  \label{co:codimensiontwopuretypeII}
  Let \( (\mfg,g,J) \) be an almost Hermitian Lie algebra.
  Then \( (\mfg, g, J) \) is a two-step solvable SKT Lie algebra of
  pure type~II for which \( \derg \) is of codimension two if and only
  if there is a complex unitary basis \( Y_1,\dots,Y_s \) of
  \( (\derg,g,J) \), elements \( Z_{1},Z_{2} \in \mfg \) spanning a
  two-dimensional complement to \( \derg \), and
  \( (a_{1,1},a_{1,2}),\dots,(a_{1,s},a_{2,s})\in \bR^{2}\setminus
  \Set{(0,0)} \) such that (up to anti-symmetry and complex-linear
  extension) the only non-zero Lie brackets are given by
  \begin{equation*}
    [Z_{k},Y_j] = a_{k,j} JY_j,
  \end{equation*}
  for \( k = 1,2 \) and \( j=1,\dots,s \).
\end{corollary}

\begin{proof}
  Choose \( Z \in V_{J}\setminus\Set{0} \).
  Then \( Z,JZ \) is a basis for \( V_{J} \) and hence the union of
  the images of \( \ad(Z)|_{\derg} \) and \( \ad(JZ)|_{\derg} \)
  spans~\( \derg \).
  As \( [Z,JZ] \in \derg \) and \( \derg \)~is Abelian, we can find
  \( \tilde{Y},\hat{Y}\in \derg \) such that
  \( Z_{1} = Z + \tilde{Y} \), \( Z_{2} = JZ + \hat{Y} \) has
  \( [Z_{1},Z_{2}]=0 \).
  The result now follows from Lemma~\ref{le:SKTpuretypeII} with
  \( a_{1,j} = \alpha_{j}(Z) \) and \( a_{2,j} = \alpha_{j}(JZ) \).
\end{proof}

If \( \mfg \) is six-dimensional, one deduces that \( \mfg \) admits a
dual basis \( e^{1},\dots,e^{6} \) whose differentials are given
either by
\begin{equation}
  \label{eq:6-1}
  (25,-15,46,-36,0,0),
\end{equation}
when \( (a_{k,j}) \) is of rank two, or by
\begin{equation}
  \label{eq:6-2}
  (25,-15,\lambda.45,-\lambda.35,0,0) \qquad\text{for some
  \( \lambda\in (0,1] \),}
\end{equation}
when \( (a_{k,j}) \) has rank one.
In the first case, the Lie algebra is isomorphic to \( 2\mfr'_{3,0} \)
and in the second case to \( \mfg_{5,17}^{0,0,\lambda}\oplus \bR \).
This covers the remaining equations in \cite[Theorem~7.1]{FrSw2} and
we have

\begin{theorem}
  \label{th:classification6dSKT}
  A six-dimensional two-step solvable Lie algebra \( \mfg \) admits an
  SKT structure if and only if it is one of the algebras explicitly
  listed in \cite[Corollary 4.8, Theorems 4.10, 7.5 and~7.1]{FrSw2} or
  it is one of the algebras in \eqref{eq:6-1} or \eqref{eq:6-2}.  \qed
\end{theorem}

Returning to Corollary~\ref{co:codimensiontwopuretypeII}, we see from
the fact that \( \ad(Z+Y)|_{\derg}=\ad(Z)|_{\derg} \) for all
\( Z\in V_J \) and \( Y\in \derg \), that the SKT condition only
depends on \( g|_{\derg} \).

\begin{corollary}
  \label{co:changeSKTmetriccodim2}
  Let \( (\mfg,g,J) \) be a two-step solvable SKT Lie algebra of pure
  type~II such that \( \derg \) is of codimension two in~\( \mfg \).
  If \( \tilde{g} \) is another Hermitian metric on \( (\mfg,J) \)
  with \( \tilde{g}|_{\derg}=g|_{\derg} \), then \( \tilde{g} \) is
  also SKT. \qed
\end{corollary}

Thus, we obtain the desired result in the codimension two case:

\begin{corollary}
  \label{co:compatibilitypuretypeIIcodim2}
  Suppose \( (\mfg,J) \) is a two-step solvable Lie algebra with a
  complex structure of pure type~II such that \( \derg \) is of
  codimension two.
  If \( J \) is SKT and is balanced, the J is also K\"ahler.
\end{corollary}

\begin{proof}
  Let \( \tilde{g} \) be the SKT metric and \( \hat{g} \) the balanced
  metric.
  Write \( \hat{V}_{J} \) for the orthogonal complement of \( \derg \)
  with respect to \( \hat{g} \).
  Define a new metric \( g \) on \( \mfg \) by declaring \( \derg \)
  to be \( g \)-orthogonal~\( \hat{V}_{J} \) and setting
  \begin{equation*}
    g|_{\derg} = \tilde{g}|_{\derg},\quad
    g|_{\hat{V}_{J}} = \hat{g}|_{\hat{V}_{J}}.
  \end{equation*}
  Then \( g \) is compatible with \( (\mfg,J) \) and both balanced by
  Corollary~\ref{co:changebalancedmetric} and SKT by
  Corollary~\ref{co:changeSKTmetriccodim2}.
  Thus, \( g \) is K\"ahler by
  Proposition~\ref{pro:balanced+SKT=Kahler}.
\end{proof}

Next, we consider the general case of a two-step solvable SKT Lie
algebras and give a full classification of these, up to some
``nilpotent terms''.
This will be sufficient to obtain the generalisation of
Corollary~\ref{co:compatibilitypuretypeIIcodim2}.

\begin{theorem}
  \label{th:SKTpuretypeII}
  Let \( (\mfg,g,J) \) be a two-step solvable almost Hermitian Lie
  algebra of pure type~II.
  Then \( (\mfg,g,J) \) is SKT if and only if there exists a complex
  unitary basis of \( Y_1,\dots,Y_s \) of~\( (\derg,g,J) \) and, for
  some \( m \in \Set{0,\dots,s} \), one-forms
  \( \alpha_1,\dots,\alpha_m\in V_J^*\setminus\Set{0} \), numbers
  \( z_1,\dots,z_m\in \bC \), complex \( (1,1) \)-forms
  \( \varphi_{m+1},\dots,\varphi_s \in \Lambda^{1,1} V_J^* \) and
  complex \( (2,0) \)-forms
  \( \psi_{m+1},\dots,\psi_s\in \Lambda^{2,0} V_J^* \) such that
  \begin{equation*}
    \sum_{k=m+1}^s \varphi_{k} \wedge \overline{\varphi_{k}}
    - \psi_k \wedge \overline{\psi_k} = 0,
  \end{equation*}
  the two forms \( \varphi_{k}+\psi_{k} \), \( k = m+1,\dots,s \), are
  linearly independent, and the only non-zero Lie brackets (up to
  anti-symmetry and complex-linear extension) are given by
  \begin{gather}
    \label{eq:Z-Yj-m}
    [Z,Y_j] = \alpha_j(Z)\, JY_j, \qquad\text{\( j = 1,\dots,m\),}\\
    [Z,W] =\sum_{j=1}^{m} z_j\, (\alpha_j\wedge J^*\alpha_j)(Z,W)\,
    Y_j + \sum_{k=m+1}^{s} \paren{\varphi_{k} + \psi_k} (Z,W)\, Y_k
  \end{gather}
  for all \( Z,W\in V_J \).
\end{theorem}

In the above, we have used complex notation, so
\( (x+iy)Z = xZ+yJZ \), etc.

\begin{proof}
  Use Lemma~\ref{le:SKTpuretypeII} to choose a complex unitary basis
  \( Y_1,\dots,Y_s \) so that \eqref{eq:Z-Yj-m} holds, with
  \( \alpha_{1},\dots,\alpha_{m} \) non-zero, and \( [Z,Y_{k}] = 0 \),
  for \( k > m \).  Using shear data, define
  \begin{equation*}
    \nu \coloneqq \omega(\any,\any)|_{\Lambda^2 U_J}\in \Lambda^2
    U_J^*\otimes \mfa = \Lambda^2 U_J^*\otimes \mfa_{J}.
  \end{equation*}
  In complex notation, we may write
  \( \nu = \sum_{j=1}^{s} \nu_j Y_j \) with
  \( \nu_j\in \Lambda^2 U_J^*\otimes \bC \).
  The first equation of~\eqref{eq:complexsheardata} yields
  \begin{equation*}
    \begin{split}
      0&= \sum_\cycl \omega(\omega(Z_1,Z_2),Z_3)
         = \sum_\cycl \omega \paren[\bigg]{ \sum_{j=1}^s
         \nu_j(Z_1,Z_2)Y_{j}, Z_3 }\\
       &= \sum_{j=1}^m\sum_\cycl  \alpha_j(Z_3)\nu_j(Z_1,Z_2) JY_j
         = \sum_{j=1}^m (\alpha_j\wedge\nu_j)(Z_1,Z_2,Z_3) JY_j,
    \end{split}
  \end{equation*}
  for all \( Z_{k} \in U_J \).
  Hence, there exist complex one-forms
  \( \zeta_{j} \in U_J^*\otimes\bC \), for \( j = 1,\dots, m \), such
  that
  \begin{equation}
    \label{eq:nu-al-ze}
    \nu_j=\alpha_j\wedge \zeta_{j}.
  \end{equation}
  Evaluating \eqref{eq:SKTsheardata} on \( Z_1,Z_2,Z_3\in U_J \) and
  some \( Y \in \spa{Y_{j},JY_{j}} \) gives
  \begin{equation}
    \label{eq:Z3Y}
    \begin{split}
      0&=\sum_\cycl
         \begin{aligned}[t]
           &g(\omega(JZ_1,JZ_2),\omega(Z_3,Y)) +
           g(\omega(JZ_3,JY),\omega(Z_1,Z_2)) \\
           &\qquad +g(\omega(J\omega(Z_1,Z_2),JZ_3),Y).
         \end{aligned}
      \\
       &= \sum_\cycl \sum_{k=1}^{s}
         \begin{aligned}[t]
           & - g\paren[\big]{J^*\nu_k(Z_1,Z_2) Y_k, \alpha_j(Z_3)JY} +
           g\paren[\big]{\alpha_{j}(JZ_{3})Y,
           \nu_{k}(Z_{1},Z_{2})Y_{k}} \\
           &\qquad -
           g\paren[\big]{\alpha_{k}(JZ_{3})\nu_{k}(Z_{1},Z_{2})Y_{k},
           Y}
         \end{aligned}
      \\
       &= - \sum_\cycl \alpha_j(Z_3) g\paren[\big]{J^*\nu_j(Z_1,Z_2)
         Y_j, JY}.
    \end{split}
  \end{equation}
  This holds trivially for \( j > m \).
  For \( j \leqslant m \) it gives
  \begin{equation*}
    \paren[\big]{\alpha_{j}\wedge J^{*}\alpha_{j} \wedge
    g(J^{*}\zeta_{j}(\any)Y_{j},JY)}(Z_{1},Z_{2},Z_{3}) = 0.
  \end{equation*}
  Taking \( Y = Y_{j} \) and then \( Y = JY_{j} \) we get that
  \( \alpha_{j}\wedge J^{*}\alpha_{j} \wedge J^{*}\zeta_{j} = 0 \).
  So \( \zeta_{j} \in \spa{\alpha_{j},J^{*}\alpha_{j}} \) and
  \( \nu_{j} = - z_{j}\alpha_{j}\wedge J^{*}\alpha_{j} \) for some
  \( z_{j} \in \bC \).

  In complex notation the second equation
  of~\eqref{eq:complexsheardata} is \( J^*\nu_j=\nu_j-i\, J.\nu_j \).
  This says that the \( (0,2) \)-part of \( \nu_{j} \) vanishes.
  For \( j \leqslant m \), we already have that \( \nu_{j} \) is
  type~\( (1,1) \).
  For \( j > m \), we write \( \nu_{j} = -\varphi_{j} - \psi_{j} \)
  with \( \varphi_{j} \) type \( (1,1) \) and \( \psi_{j} \) type
  \( (2,0) \).

  Now the only remaining equation to satisfy
  is~\eqref{eq:SKTsheardata} evaluated on \( \Lambda^4 U_J \).
  In this case, only the first term of~\eqref{eq:SKTsheardata}
  contributes, since \( \im(J^{*}\omega) = \mfa \perp U_{J} \), so we
  have
  \begin{equation*}
    \begin{split}
      0
      &= \cA(g(J^*\omega(\any,\any),
        \omega(\any,\any)))|_{\Lambda^4 U_J}
        = \sum_{j=1}^s \re(J^*\nu_j \wedge \overline{\nu_{j}})) \\
      &=\sum_{k=m+1}^s \re\paren[\big]{ (\varphi_{k} -
        \psi_{k})\wedge \overline{(\varphi_{k} + \psi_{k})}}
        =\sum_{k=m+1}^s \varphi_{k} \wedge \overline{\varphi_{k}} -
        \psi_j\wedge\overline{\psi_j},
    \end{split}
  \end{equation*}
  and the claimed result.
\end{proof}

As the metric on \( V_{J} \) plays no role in
Theorem~\ref{th:SKTpuretypeII}, we get the following version of
Corollary~\ref{co:changeSKTmetriccodim2} in arbitrary codimension.

\begin{corollary}
  \label{co:changeSKTmetric}
  Let \( (\mfg,g,J) \) be a two-step solvable SKT Lie algebra of pure
  type~II.
  If \( \tilde{g} \) is another Hermitian metric on \( (\mfg,J) \)
  with \( \tilde{g}|_{\derg}=g|_{\derg} \) and
  \( \derg^{\perp_{\tilde{g}}}=\derg^{\perp_g} \), then
  \( \tilde{g} \) is SKT too.
\end{corollary}

Moreover, we may also change an SKT metric on a two-step solvable SKT
Lie algebra of pure type~II in such a way that
Theorem~\ref{th:SKTpuretypeII} holds with \( z_1=\dots=z_k=0 \).

\begin{proposition}
  \label{pro:changeSKTmetric}
  Let \( (\mfg,g,J) \) be a two-step solvable SKT Lie algebra of pure
  type~II.
  Then \( (\mfg,J) \) admits a compatible SKT metric \( \tilde{g} \)
  with
  \begin{equation*}
    \derg = [\tilde{V}_{J},\derg] \oplus [\tilde{V}_J,\tilde{V}_J]
  \end{equation*}
  as a Hermitian orthogonal direct sum, where \( \tilde{V}_J \) is the
  \( \tilde{g} \)-orthogonal complement of~\( \derg \).
\end{proposition}

\begin{proof}
  We use the notation from Theorem~\ref{th:SKTpuretypeII}.
  Define an injective \( R\colon V_J\to \mfg \) by
  \begin{equation*}
    R(Z) \coloneqq Z+r(Z), \qquad\text{where \(
    r(Z) = \sum_{j=1}^{k} (\alpha_{j} - i J^{*}\alpha_{j})(Z)
    z_{j}Y_{j} \)}.
  \end{equation*}
  As \( \alpha_{j} - iJ^{*}\alpha_{j} \) is type \( (1,0) \), we have
  that \( R \) is complex linear, so
  \( \tilde{V}_{J} \coloneqq R(V_{J})\) is a \( J \)-invariant
  complement to \( \derg \) in~\( \mfg \).
  We get a Hermitian metric \( \tilde{g} \) on~\( \mfg \) by declaring
  \( \derg \) to be to be \( \tilde{g} \)-orthogonal to
  \( \tilde{V}_{J} \), letting \( \tilde{g} \) be \( g \) on
  \( \derg \) and setting
  \( \tilde{g}|_{\tilde{V}_J} \coloneqq (R^{-1})^*(g|_{V_J}) \).

  For \( Z \in V_{J} \) and \( Y \in \derg \), we have
  \( [R(Z),Y]=[Z+r(Z),Y]=[Z,Y] \), so
  \( [R(Z),Y_j]=i\alpha_j(Z)\, Y_j \) for \( j \leqslant m \) and
  \( [R(Z),Y_k]=0 \) for \( k > m \).  Moreover, we have
  \begin{equation*}
    \begin{split}
      \MoveEqLeft\relax
      [r(Z),W] + [Z,r(W)]
      =\sum_{j=1}^{m} (\alpha_{j}-iJ^*\alpha_{j})(Z)
      z_{j}[Y_{j},W] + (\alpha_{j}-i J^*\alpha_{j})(W)z_{j}
      [Z,Y_{j}]\\
      &=\sum_{j=1}^{m} ((\alpha_j-i\,
        J^*\alpha_j)\wedge \alpha_j)(Z,W) z_{j} JY_j
        =-\sum_{j=1}^{m} (z_j\, \alpha_j\wedge J^*\alpha_j)(Z,W)\,
        Y_j.
    \end{split}
  \end{equation*}
  So
  \begin{equation*}
    \begin{split}
      \MoveEqLeft\relax
      [R(Z),R(W)]
      =[Z,W] + [r(Z),W] + [Z,r(W)]\\
      &=\sum_{k=m+1}^s \paren[\big]{\varphi_{k} + \psi_{k}} (Z,W)\,
        Y_k
        =\sum_{k=m+1}^s
        \paren[\big]{\tilde{\varphi}_{k} + \tilde{\psi}_{k}}
        (R(Z),R(W))\, Y_k,\\
    \end{split}
  \end{equation*}
  where \( \tilde{\varphi}_{k} = (R^{-1})^{*}\varphi_{k} \) and
  \( \tilde{\psi}_{j} = (R^{-1})^{*}\psi_{j} \).
  We may now apply Theorem~\ref{th:SKTpuretypeII}, to conclude that
  \( \tilde{g} \) is SKT.
  The non-vanishing of the \( \alpha_{j} \), \( j=1,\dots,m \), gives
  \( [\tilde{V}_{J},\derg] = \spa{Y_{1},\dots,Y_{m}} \), and the
  linear independence of \( \tilde{\varphi}_{k} + \tilde{\psi}_{k} \),
  \( k = m+1,\dots,s \), implies
  \( [\tilde{V}_{J},\tilde{V}_{J}] = \spa{Y_{m+1},\dots,Y_{s}} \), so
  these two spaces are orthogonal.
\end{proof}

These preparations now allow us to prove

\begin{theorem}
  \label{th:compatibilitypuretypeII}
  Let \( (\mfg, J) \) be a unimodular two-step solvable Lie
  algebra~\( \mfg \) with complex structure~\( J \) of pure type II
  that is SKT and is balanced.  Then \( (\mfg, J) \) is K\"ahler.
\end{theorem}

\begin{proof}
  Let \( \tilde{g} \) be an SKT metric and \( \hat{g} \) be a balanced
  metric, both compatible with \( (\mfg,J) \).
  By Proposition~\ref{pro:changeSKTmetric}, we may assume that
  \( \derg \) splits as an \( \tilde{g} \)-orthogonal direct sum of
  the complex spaces \( [\tilde{V}_{J},\derg] \) and
  \( [\tilde{V}_{J},\tilde{V}_{J}] \).
  Let \( \hat{V}_J \) be the \( \hat{g} \)-orthogonal complement to
  \( \derg \).
  Then there is a complex vector space isomorphism
  \( R\colon \tilde{V}_J\to \hat{V}_J \) of the form
  \( R(Z) = Z + r(Z) \) with
  \( r\colon \tilde{V}_J\rightarrow \derg \) complex linear.

  We define a new metric \( g \) on~\( \mfg \) by requiring
  \( \derg \) to be \( g \)-orthogonal to \( \tilde{V}_{J} \), putting
  \( g \) to be \( \tilde{g} \) on \( \derg \) and letting \( g \) on
  \( \tilde{V}_{J} \) be \( R^{*}(\hat{g}|_{\hat{V}_{J}}) \).
  This metric \( g \) is Hermitian and, by
  Corollary~\ref{co:changeSKTmetric}, SKT.

  To show that \( g \) is also balanced, recall
  Proposition~\ref{pro:balanced}, which for pure type~II implies the
  existence of a \( (\hat{g},J) \)-unitary basis
  \( \hat{Z}_1,\dots,\hat{Z}_{2\ell} \) of \( \hat{V}_J \) with
  \begin{equation*}
    \hat{C} = \sum_{j=1}^{\ell} [\hat{Z}_{2j-1},\hat{Z}_{2j}] = 0.
  \end{equation*}
  Defining \( Z_j\in \tilde{V}_J \) by
  \( Z_{j} = R^{-1}(\hat{Z}_{k}) \), we get a unitary basis
  for~\( (\tilde{V}_J,g,J) \).
  Let \( C = \sum_{j=1}^{\ell} [Z_{2j-1},Z_{2j}] \) which lies in
  \( [\tilde{V}_{J},\tilde{V}_{J}] \subset \derg \).
  As \( \derg \) is the \( g \)-orthogonal direct sum of
  \( [\tilde{V}_{J},\tilde{V}_{J}] \) and \( [\tilde{V}_{J},\derg] \),
  we have for \( Y \in [\tilde{V}_{J},\tilde{V}_{J}] \), that
  \begin{equation*}
    \begin{split}
      0&=g\paren{Y,\hat{C}}
         =\sum_{j=1}^{\ell} g(Y,[R(Z_{2j-1}),R(Z_{2j})]) \\
       &=\sum_{j=1}^\ell g(Y,[Z_{2j-1},Z_{2j}]) +
         g(Y,[Z_{2j-1},r(Z_{2j})]) + g(Y,[r(Z_{2j-1}),Z_{2j}])\\
       &=\sum_{j=1}^\ell g(Y,[Z_{2j-1},Z_{2j}]) = g(Y, C).
    \end{split}
  \end{equation*}
  We conclude that \( C = 0 \).  As \( \mfg \) is unimodular,
  Corollary~\ref{co:balancedunimodular} implies that \( g \)~is also
  balanced.
  By Proposition~\ref{pro:balanced+SKT=Kahler}, we learn that
  \( g \)~is K\"ahler.
\end{proof}

\subsection{Pure type~III}
\label{subsec:puretypeIII}

Pure type~III means that \( V_{J} = 0 \), so
\( \mfg = \derg + J\derg \).

\begin{theorem}
  \label{th:compatibilitypuretypeIII}
  Let \( (\mfg, J) \) be a unimodular two-step solvable Lie
  algebra~\( \mfg \) with complex structure~\( J \) of pure type~III.
  Then \( (\mfg,J) \) cannot be both SKT and balanced.
\end{theorem}

For the proof, we first need to recall some facts on two-step solvable
SKT Lie algebras from our previous paper.
In particular~\cite[Proposition 3.8 and Corollary~4.4]{FrSw2}
give:

\begin{lemma}
  \label{le:twostepSKT}
  Let \( (\mfg,g,J) \) be a two-step solvable SKT Lie algebra.  Then:
  \begin{enumerate}
  \item There exists a complex unitary basis \( Y_1,\dots,Y_s \) of
    \( \derg_J \) and complex-valued one-forms
    \( \xi_j\in (\derg_r)^{*}\otimes \bC \) such that
    \begin{equation*}
      [JX,Y_j] = \xi_j(X)\, Y_j
    \end{equation*}
    holds for all \( X\in \derg_r \) and all \( j\in \{1,\dots,s\} \).
  \item If \( \mfg = \derg + J\derg \), then there exists an
    \( X_0\in J\derg_r \) such that
    \begin{equation*}
      [X_0,X]-X \in \derg_J
    \end{equation*}
    for all \( X \in \derg_r \). \qed
  \end{enumerate}
\end{lemma}

\begin{proof}[Proof of Theorem~\ref{th:compatibilitypuretypeIII}]
  Suppose on the contrary that \( (\mfg, J) \) admits a Hermitian
  metric \( \tilde{g} \) that is SKT and a Hermitian metric \( g \)
  that is balanced.
  As usual \( \derg_r \) is the \( g \)-orthogonal complement of
  \( \derg_J \) in \( \derg \) and
  \( V_{r} = \derg_{r} \oplus J\derg_{r} \).
  Write \( \derg_{\tilde{r}} \) for the \( \tilde{g} \)-orthogonal
  complement of \( \derg_J \) in~\( \derg \), and put
  \( V_{\tilde{r}} = \derg_{\tilde{r}} \oplus J\derg_{\tilde{r}} \).

  Corollary~\ref{co:balancedunimodular} gives a \( (g,J) \)-unitary
  basis \( X_1,\dots,X_{2r} \) of~\( V_r \) with
  \begin{equation}
    \label{eq:balancedunimodulartypeIII}
    C = \sum_{k=1}^r [X_{2k-1},X_{2k}] = 0.
  \end{equation}
  As \( \mfg = \derg_{J} + \derg_{\tilde{r}} + J\derg_{\tilde{r}} \),
  we may write each \( X \in V_{r} \) as
  \( X = \tilde{Y}+\tilde{W}+J\tilde{X} \) with
  \( \tilde{Y} \in \derg_{J} \) and
  \( \tilde{W},\tilde{X} \in \derg_{\tilde{r}} \).

  By Lemma~\ref{le:twostepSKT}(a) there is a complex unitary basis
  \( Y_1,\dots,Y_s \) of \( (\derg_J,\tilde{g},J) \) and
  \( \xi_1,\dots,\xi_s\in (\derg_{\tilde{r}})^{*}\otimes \bC \) such
  that
  \begin{equation*}
    [X,Y_j]=[J\tilde{X},Y_j]=\xi_j(\tilde{X}) \, Y_j
  \end{equation*}
  for each \( j\in \{1,\dots,s\} \) and \( X \in V_{r} \).
  Inserting now \( X=X_{2k-1},JX = X_{2k},Y = Y_j,JY \) into the
  version of equation~\eqref{eq:SKTsheardata} for the SKT metric
  \( \tilde{g} \), and writing \( z = \xi_{j}(\tilde{X}) \),
  \( w=\xi_{j}(\tilde{JX}) \), yields
  \begin{equation*}
    \begin{split}
      0
      &=-\tilde{g}([JX,JY],[JX,JY])
        +\tilde{g}([JX,JJY],[JX,Y])
        +\tilde{g}([JJX,JY],[X,JY]) \\
      &\qquad
        -\tilde{g}([JJX,JJY],[X,Y])
        +\tilde{g}([J[X,JX],JY],JY)
        -\tilde{g}([J[X,JX],JJY],Y) \\
      &\qquad
        -\tilde{g}([J[X,Y],JJX],JY)
        +\tilde{g}([J[X,JY],JJX],Y)
        +\tilde{g}([J[JX,Y],JX],JY) \\
      &\qquad
        -\tilde{g}([J[JX,JY],JX],Y)
      \\
      &=-2\tilde{g}(w Y, w Y)
        -2\tilde{g}(z Y, z Y)
        +\tilde{g}([J[X,JX],JY],JY)
        +\tilde{g}([J[X,JX],Y],Y)   \\
      &\qquad
        -2\tilde{g}(z^{2}Y,Y)
        -2\tilde{g}(w^{2}Y,Y) \\
      &= -2 (z\overline{z} + \re(z^{2}) + w\overline{w} + \re(w^{2}))
        +\tilde{g}([J[X,JX],JY],JY) \\
      &\qquad
        +\tilde{g}([J[X,JX],Y],Y)\\
      &=-4(\re(z)^{2}+\re(w)^{2}) +\tilde{g}([J[X,JX],JY],JY)
        +\tilde{g}([J[X,JX],Y],Y)\\
      &=-4(\re(\xi_{j}(\tilde{X}_{2k-1}))^{2}
        +\re(\xi_{j}(\tilde{X}_{2k}))^{2})
      \\
      &\qquad
        +\tilde{g}([J[X_{2k-1},X_{2k}],JY_{j}],JY_{j})
        +\tilde{g}([J[X_{2k-1},X_{2k}],Y_{j}],Y_{j}).
    \end{split}
  \end{equation*}
  Summing over \( k \) and using
  equation~\eqref{eq:balancedunimodulartypeIII}, we get
  \begin{equation*}
    0=\sum_{k=1}^r\paren[\big]{\re(\xi_j(\tilde{X}_{2k-1}))^2
    +(\re(\xi_j(\tilde{X}_{2k}))^2}.
  \end{equation*}
  Thus, \( \re(\xi_j(\tilde{X}_t))=0 \) for all \( t \), so
  \( \re(\xi_{j}(\tilde{X})) = 0 \) for all \( X \in V_{r} \) and
  all~\( j \).

  Note that
  \begin{equation*}
    \begin{split}
      \tr(\ad(X)|_{\derg_{J}})
      &= \sum_{j=1}^s \tilde{g}([X,Y_j],Y_j)+\tilde{g}([X,JY_j],JY_j)
      \\
      &= \sum_{j=1}^{s} 2\tilde{g}(\xi_j(\tilde{X})Y_j,Y_j)
        = 2 \sum_{j=1}^{s} \re(\xi_{j}(\tilde{X})) = 0.
    \end{split}
  \end{equation*}
  As
  \( \im(\ad(X))\subseteq \derg = \derg_J\oplus \derg_{\tilde{r}} \),
  the unimodularity of \( \mfg \) gives for \( X \in V_{r} \) that
  \begin{equation*}
    0=\tr(\ad(X)) = \tr(\ad(X)|_{\derg_{J}}) +
    \tr(\ad(X)|_{\derg_{\tilde{r}}})
    = \tr(\ad(X)|_{\derg_{\tilde{r}}}).
  \end{equation*}
  By Lemma~\ref{le:twostepSKT}, there exists an
  \( \tilde{X}_0\in J\derg_{\tilde{r}} \) with
  \begin{equation*}
    [\tilde{X}_0,\tilde{X}]-\tilde{X}\in\derg_J
  \end{equation*}
  for any \( \tilde{X}\in\derg_{\tilde{r}} \).  Write
  \begin{equation*}
    \tilde{X}_0 = X_0 + Y_0
  \end{equation*}
  for \( X_0\in V_r \) and \( Y_0\in \derg_J \).
  Then \( \ad(Y_{0}) = 0 \) on \( \derg \) and
  \begin{equation*}
    0 = \tr(\ad(X_{0})|_{\derg_{\tilde{r}}})
    = \tr(\ad(\tilde{X}_{0})|_{\derg_{\tilde{r}}}).
  \end{equation*}
  Choosing a \( \tilde{g} \)-orthonormal basis
  \( \tilde{S}_1,\dots,\tilde{S}_r \) of \( \derg_{\tilde{r}} \), we
  have
  \begin{equation*}
    0 = \tr(\ad(\tilde{X}_{0})|_{\derg_{\tilde{r}}})
    = \sum_{k=1}^r\tilde{g}([\tilde{X}_0,\tilde{S}_k],\tilde{S}_k)
    = \sum_{k=1}^r\tilde{g}(\tilde{S}_k,\tilde{S}_k)
    = \dim(\derg_{\tilde{r}}).
  \end{equation*}
  Thus \( \derg_r=0 \) and \( \mfg=\derg+J\derg=\derg \),
  contradicting that \( \mfg \) is solvable.
\end{proof}

We end this section by providing an example that the unimodular
condition in Theorem~\ref{th:compatibilitypuretypeIII} is necessary,
and which also supports the need for this condition in
Question~\ref{quest:G}.

\begin{example}
  \label{ex:counterexampletypeIII}
  Let \( \mfg \) be the six-dimensional Lie algebra with basis
  \( e_{1},\dots,e_{6} \) whose dual basis \( e^1,\dots,e^6 \) has
  differentials given by
  \begin{equation*}
    (-15+16,-25+26,2.(35+46),2.(36+45),0,0),
  \end{equation*}
  which is isomorphic to \( N_{6,1}^{-1/2,-1/2,0,0} \).
  Consider the almost complex structure \( J \) on~\( \mfg \) given by
  \( Je_1=e_2 \), \( Je_3=e_5 \), \( Je_4=e_6 \).
  Then \( J \) is integrable, so defines a complex structure
  on~\( \mfg \).
  We have \( \derg=\spa{e_1,\dots,e_4} \) and
  \( \derg +J\derg = \mfg \).

  Consider the metric \( \tilde{g} \) on \( \mfg \) for which
  \( e_1,\dots,e_6 \) is orthonormal.
  Then \( \tilde{g} \) is compatible with \( J \) and the associated
  fundamental two-form \( \tilde{\sigma} \) is
  \begin{equation*}
    \tilde{\sigma}=e^{12}+e^{35}+e^{46}.
  \end{equation*}
  A direct computation yields
  \( dJ^*d\tilde{\sigma} = 2dJ^{*}(e^{12}\wedge(e^{5}-e^{6})) = 2
  d(e^{12}\wedge (e^{3}-e^{4})) = 0 \), so \( \tilde{g} \) is SKT.

  Next, consider the metric \( \hat{g} \) for which
  \( e_1,e_2,e_3,e_5,e_3+e_4,e_5+e_6 \) is an orthonormal basis.
  Since this basis is unitary, \( \hat{g} \) is compatible
  with~\( J \).
  As \( e^1,e^2,e^3-e^4,\allowbreak e^5-e^6,e^4,e^6 \) is the dual of
  the above basis, the associated fundamental two-form
  \( \hat{\sigma} \) is given by
  \begin{equation*}
    \hat{\sigma}=e^{12}+(e^3-e^4)\wedge
    (e^5-e^6)+e^{46}=e^{12}+e^{35}+2 e^{46}-e^{36}-e^{45}.
  \end{equation*}
  One computes
  \begin{equation*}
    d(\hat{\sigma}^2)
    = 2 \hat{\sigma} \wedge d\hat{\sigma}
    = 4\hat{\sigma} \wedge (e^{12}\wedge(e^{5}-e^{6}) + e^{456})
    = 0,
  \end{equation*}
  so \( \hat{g} \) is a balanced metric.

  Thus, \( (\mfg,J) \) is a two-step solvable Hermitian Lie algebra of
  pure type~III that is SKT and is balanced.

  We claim that \( (\mfg,J) \) is not K\"ahler.
  For contradiction, suppose \( g \) is a compatible K\"ahler metric.
  Let \( \derg_r \) be the orthogonal complement of
  \( \derg_J = \spa{e_{1},e_{2}}\) in~\( \derg \).
  Then \( J\derg_r \) is a complement of
  \( \derg=\spa{e_1,\dots,e_4} \) in \( \mfg \) and so it has to
  contain a vector of the form \( e_5+W \) for some~\( W\in \derg \).
  Moreover, by Corollary~\ref{co:Kahlerpuretype}(III), one has
  \( \tr(\ad(V)|_{\derg_J})=0 \) for any \( V\in J\derg_r \).
  However, \( \ad(e_5+W)(e_i)=[e_5,e_i]=e_i \) for \( i=1,2 \) and so
  \( \tr(\ad(e_5+W)|_{\derg_J})=2 \), a contradiction.
\end{example}

\subsection{Dimension \texorpdfstring{\( 6 \)}{6}}
\label{subsec:dimension6}

We can now consider general unimodular six-dimensional two-step
solvable Lie algebras \( \mfg \) endowed with a complex
structure~\( J \).

\begin{theorem}
  \label{th:compatibility6d}
  Let \( \mfg \) be a six-dimensional unimodular two-step solvable Lie
  algebra endowed with a complex structure \( J \).
  If \( (\mfg,J) \) is SKT and is balanced, then it is also K\"ahler.
\end{theorem}

Theorems~\ref{th:compatibilitypuretypeI},
\ref{th:compatibilitypuretypeII} and~\ref{th:compatibilitypuretypeIII}
give the result when \( (\mfg,J) \) is of pure type.
However, in dimension~\( 6 \), if \( (\mfg,J) \) is not of pure type,
then we have \( \dim(\derg_J)=2 \), \( \dim(\derg_r)=1 \) and
\( \dim(V_J)=2 \).
The SKT Lie algebras of this type are described in detail
in~\cite[Theorem 7.5]{FrSw2}.
There are three cases, but they share common properties, so that the
following holds.

\begin{proposition}
  \label{pro:SKT6dnonpuretype}
  Let \( (\mfg,g,J) \) be a six-dimensional two-step solvable SKT Lie
  algebra which is not of pure type.
  Then there exist \( Y\in \derg_J \), \( X\in \derg_r \) and
  \( Z\in V_J \), all non-zero, such that \( \ad(Z) \), \( \ad(JZ) \)
  preserve~\( \derg_J \) and are complex-linear on that space.
  Additionally there exist
  \( (b_{0},b_1,b_2,b_{3})\in \bR^4\setminus \{0\} \) and
  \( z_{0},z_1,z_2,w_{0},\dots,w_{5}\in \bC \), with
  \begin{equation*}
    b_{0}b_{3} + b_1^2 + b_2^2 = 0,\qquad
    \re(z_{i}) = -\delta_{i}b_i/2,\quad \text{for \( i=0,1,2 \) and
    some \( \delta_{i} \in \Set{0,1} \)},
  \end{equation*}
  and with \( z_{0}=0 \) implying \( b_{0}=b_1=b_2=0 \), such that the
  only non-zero Lie brackets (up to anti-symmetry and complex-linear
  extension) are given by
  \begin{gather*}
    [JX,Y] = z_{0} Y,\quad [Z,Y] = z_1 Y,\quad [JZ,Y] = z_2 Y,\\
    [JX,X] = b_{0} X+w_{0} Y,\quad [Z,X] = b_1 X+ w_1 Y,\quad
    [JZ,X] = b_2 X+ w_2 Y,\\
    [Z,JX] = -b_2 X+ w_3 Y,\quad [JZ,JX] = b_1 X+w_4 Y,\quad [Z,JZ] =
    b_{3} X+ w_{5} Y.
  \end{gather*}
  \qed
\end{proposition}

Theorem~\ref{th:compatibility6d} now follows from the following result
that does not require \( \mfg \) to be unimodular.

\begin{proposition}
  \label{pro:compatibility6dnonpuretype}
  Let \( \mfg \) be a six-dimensional Lie algebra endowed with a
  complex structure \( J \) such that \( (\mfg,J) \) is not of pure
  type.
  If \( (\mfg,J) \) is SKT and is balanced, then it also K\"ahler.
\end{proposition}

\begin{proof}
  We use the notation of Proposition~\ref{pro:SKT6dnonpuretype}.
  Moreover, let \( \tilde{g} \) be a compatible balanced metric
  on~\( \mfg \), write \( \tilde{V}_r \) for the
  \( \tilde{g} \)-orthogonal complement of \( \derg_J \) in
  \( \derg+J\derg \) and \( \tilde{V}_J \) for the
  \( \tilde{g} \)-orthogonal complement of \( \derg+J\derg \)
  in~\( \mfg \).

  Choose a \( \tilde{g} \)-unit vector
  \( \tilde{X} \in \tilde{V}_r\cap\derg \).
  Then \( \tilde{X} \) has the form
  \begin{equation*}
    \tilde{X} = \mu_{0} X + \tilde{Y}
  \end{equation*}
  for some \( \mu_{0}\in \bR\setminus \{0\} \) and
  \( \tilde{Y}\in \derg_J \).
  We may also find a \( \tilde{g} \)-unit vector
  \( \tilde{Z} \in \tilde{V}_J \) of the form
  \begin{equation*}
    \tilde{Z}=\mu_1 Z+\mu_2 JX+\mu_3 X+\hat{Y}
  \end{equation*}
  with \( \mu_1\in \bR\setminus \{0\} \), \( \mu_2,\mu_3\in \bR \)
  and~\( \hat{Y}\in \derg_J \).
  Proposition~\ref{pro:balanced} has
  \( C = [\tilde{X},J\tilde{X}]+[\tilde{Z},J\tilde{Z}] \) and implies
  \begin{equation}
    \label{eq:tracecondnonpuretype}
    \tr(\ad(J\tilde{X}))
    = - \tilde{\sigma}(C,J\tilde{X})
    = -\tilde{g}([\tilde{X},J\tilde{X}]+[\tilde{Z},J\tilde{Z}],
    \tilde{X}).
  \end{equation}

  We have
  \( \tr(\ad(JX)) = 2\re(z_{0}) + b_{0} = (1-\delta_{0})b_{0} \) and
  hence
  \begin{equation*}
    \tr(\ad(J\tilde{X})) = \mu_{0} \tr(\ad(JX)) = \mu_{0}
    (1-\delta_{0})b_{0}.
  \end{equation*}
  On the other hand, \( \tilde{X} \) is \( \tilde{g} \)-orthogonal
  to~\( \derg_{J} \), so
  \begin{equation*}
    \begin{split}
      \MoveEqLeft
      \tilde{g}([\tilde{X},J\tilde{X}]+[\tilde{Z},J\tilde{Z}],
      \tilde{X})\\
      &=\tilde{g}([\mu_{0} X,\mu_{0} JX]+[\mu_1 Z+\mu_2 JX+\mu_3
        X,\mu_1 JZ-\mu_2 X+\mu_3 JX],\tilde{X})\\
      &=\tilde{g}\paren[\big]{\paren[\big]{-b_{0} (\mu_{0}^2
        + \mu_2^2 + \mu_3^2) + \mu_1^2b_{3}
        - 2\mu_1 \mu_2 b_1 - 2\mu_1\mu_3 b_2}X,
        \tilde{X}}\\
      &=\frac{1}{\mu_{0}}\paren[\big]{-b_{0}(\mu_{0}^2 + \mu_2^2 +
        \mu_3^2) + \mu_1^2b_{3} -2\mu_1 (\mu_2 b_1+\mu_3 b_2)}.
    \end{split}
  \end{equation*}
  Putting these expressions in to
  equation~\eqref{eq:tracecondnonpuretype} times \( \mu_{0} \), we get
  \begin{equation}
    \label{eq:conditionscoeffs}
    b_{0}(\delta_{0}\mu_{0}^2+\mu_2^2+\mu_3^2) - \mu_1^2 b_{3}
    + 2\mu_1(\mu_2 b_1+\mu_3 b_2) = 0.
  \end{equation}

  Recall that \( b_{0}b_{3} + b_{1}^{2} + b_{2}^{2} = 0 \).
  If \( b_{0}=0 \), then \( b_1=b_2=0 \), so
  \eqref{eq:conditionscoeffs} gives \( b_{3} = 0 \), which contradicts
  \( (b_{0},b_{1},b_{2},b_{3}) \) being non-zero.
  Thus \( b_{0} \ne 0 \).
  Now multiplying equation~\eqref{eq:conditionscoeffs} by \( b_{0} \)
  gives
  \begin{equation*}
    \begin{split}
      0&=b_{0}^{2}(\delta_{0}\mu_{0}^2+\mu_2^2+\mu_3^2)
         +(b_1^2+b_2^2)\mu_1^2+2\mu_1b_{0}(
         \mu_2 b_1+\mu_3 b_2)\\
       &= b_{0}^2\delta_{0} \mu_{0}^2+ (\mu_1b_1 + \mu_2b_{0})^2
         +(\mu_1b_2 + \mu_3b_{0})^2.
    \end{split}
  \end{equation*}
  As \( \mu_{0}\neq 0 \), we thus have
  \begin{equation*}
    \delta_{0} = 0,\qquad
    \mu_{1}b_{1} + \mu_{2}b_{0} = 0 \quad\text{and}\quad
    \mu_{1}b_{2} + \mu_{3}b_{0} = 0.
  \end{equation*}
  Using this, and since the structure equations give
  \( \tr(\ad(Z))=(1-\delta_1)b_1 \) and
  \( \tr(\ad(JZ))=(1-\delta_2)b_2 \), one computes
  \begin{equation*}
    \begin{split}
      \tr(\ad(\tilde{Z}))
      &= \mu_1 \tr(\ad(Z))+\mu_2 \tr(\ad(JX))
        = \mu_1(1-\delta_1)b_1 + \mu_2 b_{0} \\
      &= -\delta_1 \mu_1 b_1 = 2\mu_{1} \re(z_{1}),
    \end{split}
  \end{equation*}
  and similarly
  \( \tr(\ad(J\tilde{Z})) = -\delta_{2}\mu_{1}b_{2} =
  2\mu_{1}\re(z_{2}) \).
  By Proposition~\ref{pro:balanced}, we must have
  \( \tr(\ad(\tilde{Z}))=\tr(\ad(J\tilde{Z}))=0 \), so
  \( \re(z_{i}) = 0 \) for \( i=1,2 \).
  Write \( z_{j} = ic_{j} \), for some \( c_{j}\in \bR \),
  \( j = 0,1,2 \).

  Putting \( \hat{X} \coloneqq X + (w_{0}/(1-ic_{0}))Y \), one
  computes that
  \begin{equation*}
    [J\hat{X},Y] = ic_{0} Y,\qquad [J\hat{X},\hat{X}] = b_{0} \hat{X}.
  \end{equation*}
  Note that, since \( b_{0}\neq 0 \), we must have
  \( c_{0} = -iz_{0} \neq 0 \).  Noting that
  \begin{equation*}
    [Z,\hat{X}]=b_1 \hat{X}+\hat{w}_1 Y \quad\text{and}\quad
    [JZ,\hat{X}]=b_2 \hat{X}+\hat{w}_2 Y ,
  \end{equation*}
  for some \( \hat{w}_1,\hat{w}_2\in \bC \), the Jacobi identity
  yields
  \begin{equation*}
    \begin{split}
      0&=[Z,[J\hat{X},\hat{X}]]+[J\hat{X},[\hat{X},Z]]
         =[Z,b_{0}\hat{X}]-[J\hat{X},b_1\hat{X}+\hat{w}_1
         Y]\\
       &=b_{0} b_1 \hat{X}+b_{0}\hat{w}_1Y-b_{0} b_1 \hat{X}
         - ic_{0} \hat{w}_1 Y
         =(b_{0}-ic_{0})\hat{w}_1 Y.
    \end{split}
  \end{equation*}
  As \( b_{0}\ne 0 \), we conclude that \( \hat{w}_1=0 \).
  Similarly, we obtain \( \hat{w}_2=0 \).  Thus, setting
  \begin{equation*}
    \check{Z} \coloneqq Z - \frac{b_1}{b_{0}} J\hat{X} -
    \frac{b_2}{b_{0}}\hat{X},
  \end{equation*}
  one easily checks that
  \begin{equation*}
    [\check{Z},\hat{X}]=0,\quad
    [J\check{Z},\hat{X}]=0,\quad
    [\check{Z},Y]=i\tilde{c}_1 Y,\quad
    [J\check{Z},Y]=i\tilde{c}_2 Y,
  \end{equation*}
  for some \( \tilde{c}_1,\tilde{c}_2\in \bC \) and that
  \( [\check{Z},J\check{Z}]\in \derg_J \).  Moreover, we check that
  \begin{equation*}
    [\check{Z},J\hat{X}]=\tilde{w} Y,\qquad
    [J\check{Z},J\hat{X}]=\hat{w} Y,
  \end{equation*}
  for certain \( \tilde{w},\hat{w}\in \bC \).
  Then, the vanishing of the Nijenhuis tensor \( N_J \) on
  \( (\check{Z},\hat{X}) \) yields \( \hat{w}=i\tilde{w} \).
  Hence, setting
  \begin{equation*}
    \hat{Z} \coloneqq \tilde{Z}+\frac{\tilde{w}}{ic_{0}}Y,
  \end{equation*}
  one calculates
  \begin{equation*}
    [\hat{Z},J\hat{X}]= [J\hat{Z},J\hat{X}]=0.
  \end{equation*}
  Furthermore, \( [\hat{Z},J\hat{Z}]\in \derg_J \) so the Jacobi
  identity yields
  \begin{equation*}
    0=[\hat{Z},[J\hat{Z},J\hat{X}]]
    +[J\hat{Z},[J\hat{X},\hat{Z}]]+[\hat{JX},[\hat{Z},J\hat{Z}]]
    =ic_{0} [\hat{Z},J\hat{Z}],
  \end{equation*}
  and we conclude that \( [\hat{Z},J\hat{Z}]=0 \).

  Thus, denoting the basis
  \( Y,JY,\hat{X},J\hat{X},\hat{Z},J\hat{Z} \) of \( \mfg \) by
  \( e_1,\dots,e_6 \), the differentials of the dual basis
  \( e^1,\dots,e^6 \) are given by
  \begin{equation*}
    (-c_{0}.24-\tilde{c}_1.25-\tilde{c}_2.26,
    c_{0}.14+\tilde{c}_1.15+\tilde{c}_2.16,b_{0}.34,0,0,0).
  \end{equation*}
  The metric \( g \) for which \( e_1,\dots,e_6 \) is orthonormal is
  Hermitian for~\( J \) with associated two-form
  \( \sigma=e^{12}+e^{34}+e^{56} \).
  But \( d\sigma = 0 \), so \( (\mfg,g,J) \) is K\"ahler.
\end{proof}

\begin{remark}
  \label{re:Kahler6dnonpuretype}
  From the proof of Proposition~\ref{pro:compatibility6dnonpuretype},
  one deduces that the \( 6 \)-dimensional two-step solvable Lie
  algebras which admit a K\"ahler structure of non-pure type are given
  by
  \begin{equation*}
    (-24,14,a.  34,0,0,0),\qquad (-25,15,34,0,0,0)
  \end{equation*}
  for \( a>0 \), with the second case occurring exactly when
  \( (\tilde{c}_{1},\tilde{c}_{2}) \ne 0 \).
  The algebras in the family are almost Abelian and isomorphic to
  \( \mfr'_{4,a,0}\oplus \bR^2 \).
  The second algebra is isomorphic to
  \( \mfr_{3,0}'\oplus \aff_{\bR}\oplus \bR \).
\end{remark}

From Remark~\ref{re:Kahler6dnonpuretype} and
Corollary~\ref{co:Kahlerpuretype} we deduce

\begin{corollary}
  \label{th:6dKahler}
  The six-dimensional two-step solvable Lie algebras admitting a
  K\"ahler structure are the following ones:
  \begin{gather*}
    N_{6,14}^{\alpha,\beta,0}\ (\alpha\beta\neq 0),\
    \mfg_{6,11}^{\alpha,0,0,\delta}\ (\alpha\delta\neq 0),\,
    \mfg_{5,17}^{0,0,\lambda}\oplus \bR\ (\lambda\in (0,1]),\
    \mfr'_{4,a,0}\oplus \bR^2\ (a>0),\\
    \mfr_{3,0}'\oplus \mfr_{3,0}',\ \mfr_{3,0}'\oplus \aff_{\bR}\oplus
    \bR,\ \mfr_{3,0}'\oplus \bR^3,\ 3\aff_{\bR},\ 2\aff_{\bR}\oplus
    \bR^2,\ \aff_{\bR}\oplus \bR^4,\ \bR^6.
  \end{gather*}
\end{corollary}

\begin{proof}
  By Remark~\ref{re:Kahler6dnonpuretype},
  \( \mfr'_{4,a,0}\oplus \bR^2 \) for \( a>0 \) and
  \( \mfr_{3,0}'\oplus \aff_{\bR}\oplus \bR \) are precisely the
  six-dimensional two-step solvable Lie algebras admitting a K\"ahler
  structure which are not of pure type.

  For pure type~I, Corollary~\ref{co:Kahlerpuretype}(I) gives that the
  algebras are \( k\aff_{\bR}\oplus \bR^{6-2k} \) for
  \( k\in \{0,\dots,3\} \).

  For pure type~II, we use Corollary~\ref{co:Kahlerpuretype}(II).
  This gives (a)~for \( \dim(\derg_J)=2 \), the algebra
  \begin{equation*}
    (-23,13,0,0,0,0),
  \end{equation*}
  which is \( \mfr_{3,0}'\oplus \bR^3 \) and (b)~for
  \( \dim(\derg_J)=4 \), one of
  \begin{equation*}
    (-25,15,-46,36,0,0),\qquad (-25,15,-\lambda.45,\lambda.35,0,0)
  \end{equation*}
  for \( \lambda\in (0,1] \), which are
  \( \mfr_{3,0}'\oplus \mfr_{3,0}' \) and
  \( \mfg_{5,17}^{0,0,\lambda}\oplus \bR \), respectively.

  Finally, for pure type~III, Corollary~\ref{co:Kahlerpuretype}(III)
  gives (a)~for \( \dim(\derg_J)=2 \) the algebras
  \begin{equation*}
    (-25-c.26,15+c.16,a_1.35,a_2.46,0,0)
  \end{equation*}
  for some \( a_1,a_2\in \bR\setminus \{0\} \), \( c\in \bR \), which
  are isomorphic to \( N_{6,14}^{\alpha,\beta,0} \) for certain
  \( \alpha,\beta\in \bR\setminus \{0\} \), and (b)~for
  \( \dim(\derg_J)=4 \), the algebras
  \begin{equation*}
    (-26,16,-c.46,c.36,a.56,0)
  \end{equation*}
  for certain \( a,c\in \bR\setminus \{0\} \), which are isomorphic to
  \( \mfg_{6,11}^{\alpha,0,0,\delta} \) for
  \( \alpha,\delta\in \bR\setminus \{0\} \).
\end{proof}

\providecommand{\bysame}{\leavevmode\hbox to3em{\hrulefill}\thinspace}
\providecommand{\MR}{\relax\ifhmode\unskip\space\fi MR }
\providecommand{\MRhref}[2]{%
  \href{http://www.ams.org/mathscinet-getitem?mr=#1}{#2}
}
\providecommand{\href}[2]{#2}

\end{document}